\documentclass[11pt]{article}

\usepackage{amsmath,amsthm,verbatim,amssymb,amsfonts,amscd, graphicx}
\usepackage{fancyhdr}
\usepackage{amsmath,amsthm,amsfonts,amssymb}
\usepackage{epic,eepic}
\usepackage{authblk}
\usepackage{makeidx, enumerate,hyperref}
\usepackage[active]{srcltx}
\usepackage[backend=bibtex,style=numeric]{biblatex}
\renewbibmacro{in:}{}
\bibliography{bibbase2}
\theoremstyle{plain}
\newtheorem{theorem}{Theorem}[section]
\newtheorem{proposition}[theorem]{Proposition}
\newtheorem{lemma}[theorem]{Lemma}
\newtheorem{corollary}[theorem]{Corollary}
\theoremstyle{definition}
\newtheorem{definition}[theorem]{Definition}
\theoremstyle{remark}
\newtheorem{remark}[theorem]{Remark}
\numberwithin{equation}{section} 
\numberwithin{figure}{section} 
\numberwithin{theorem}{section}
\newtheorem{example}[theorem]{Example}

\begin{document}
\begin{flushleft}
{\LARGE{\bf Second-order sensitivity relations and regularity of the value function for Mayer's problem in optimal control}}

\vskip0.5\baselineskip{\bf P. Cannarsa \footnotemark[1], H. Frankowska \footnotemark[2], T. Scarinci \footnotemark[1]\footnotemark[2]}
\end{flushleft}

\footnotetext[1]{Dipartimento di Matematica,
Universit\`a di Roma Tor Vergata, Via della Ricerca Scientifica 1, 00133 Roma, Italy. }
\footnotetext[2]{CNRS, IMJ-PRG, UMR 7586, Sorbonne Universit\'es, UPMC Univ Paris 06, UnivParis Diderot, Sorbonne Paris Cit\'e, Case 247, 4 Place Jussieu, 75252 Paris, France.}
\section*{Abstract}
This paper investigates the value function, $V$, of a Mayer optimal control problem with the state equation  given by a differential inclusion. First, we obtain an invariance property for the proximal and Fr\'echet subdifferentials of $V$ along optimal trajectories. Then, we extend the analysis to the sub/superjets of $V$, obtaining new sensitivity relations of second order. By applying sensitivity analysis to exclude the presence of conjugate points, we deduce that the value function is twice differentiable along any optimal trajectory starting at a point at which $V$ is proximally subdifferentiable. We also provide sufficient conditions for the local $C^2$ regularity of $V$ on tubular neighborhoods of optimal trajectories.

\begin{flushleft}
{\bf Keywords:} Mayer problem, differential inclusion, sensitivity relations, Riccati equation.\\
{\bf MSC Subject classifications:} 34A60, 49J53.
\end{flushleft}
\section{Introduction}
The value function $V$ of the Mayer optimal control problem plays
a fundamental role in the investigation of optimal trajectories
and optimal controls. Being the unique solution of the
Hamilton-Jacobi-Bellman equation
\begin{equation}\label{HJB}
-V_t  + H(x,-V_x )=0, \;\; V(T,\cdot)= \phi(\cdot),
\end{equation}
it was intensively studied since the late fifties. If $V$ is
differentiable, then,  given an optimal trajectory $\bar x(\cdot)$, the function $t \to p(t):=- \nabla_x V(t,\bar x(t))$ is
the adjoint state (dual arc) of the celebrated maximum principle, see for instance
\cite{MR0454768}.  Hence, \eqref{HJB} implies the following
sensitivity relation involving the dual arc  $p(\cdot)$:
\begin{equation}\label{SR}
(H(\bar x(t),p(t)), -p(t))=\nabla V(t,\bar x(t)) .
\end{equation}
However, in general,  $V$ is  not differentiable.

Since the  first generalizations of \eqref{SR} in the seminal
papers by Clarke and Vinter~\cite{Clarke:1987:RMP:35498.35509} and
Vinter~\cite{MR923279}, sensitivity analysis has been developed
for different kinds of optimal control problems, leading to
interesting applications to optimal synthesis
and regularity of the value function. Typically, sensitivity relations are given in the form of inclusions of the pair formed by the dual arc and the Hamiltonian, evaluated along an optimal trajectory,
 into a suitable generalized differential of the value function. Such a differential, which was originally identified with Clarke's generalized gradient,
 was later restricted to the Fr\'echet superdifferential in many different situations, including the Bolza, minimum time  and Mayer problems, with or without state constraints
  (see, e.g., \cite{Cannarsa:1991:COT:120771.118946,MR894990,Subb2,MR1780579,Nostro,Bettiol:2010:SIC:1958083.1958102,frankowska:hal-00800199,BetFRanVint}).
   Such restriction is of a crucial importance, because it also allows to
   get sufficient optimality conditions and describe the optimal
   synthesis, cf. \cite{MR894990,Cannarsa:1991:COT:120771.118946}. We  refer also to \cite{Goebel} for a construction
   of optimal feedback for a  convex  Bolza problem using subdifferentials of the value function.  Sensitivity
relations for nonsmooth value function of the Mayer  problem are
also useful for investigation of some geometric properties of
reachable sets, see for instance  \cite{lorenz-boundary}.

This paper, which is focussed on Mayer's problem for differential
inclusions, differs from the aforementioned references mainly in
two aspects:

1.  We concentrate on sensitivity relations for the proximal and
the Fr\'echet subdifferentials, instead of superdifferentials.

2.  We aim to obtain second-order inclusions, that is, inclusions
of suitable dual pairs into second-order subjets and superjets of
$V(t,\cdot)$.

While the validity of first-order subdifferential inclusions for the adjoint state has already been observed and exploited in \cite{frankowska:hal-00851752} for the calculus of variations and
in \cite{Cannarsa2013791} for the Bolza optimal control problem,
to the best of our knowledge this is for the first time that
second-order sensitivity relations are obtained for optimal
control problems of any kind. Such relations turn out to be very
useful for the analysis of the regularity of the value function: in
this paper we apply them to study the second-order
differentiability of $V$ along optimal trajectories.

To become more specific, given a locally Lipschitz final cost
$\phi: \mathbb{R}^n \rightarrow \mathbb{R}$, for any fixed
$(t_0,x_0)\in (-\infty,T]\times\mathbb{R}^n$ consider the Mayer
problem:
\begin{equation}\label{Mayer}
\mbox{minimize}\; \phi(x(T)),
\end{equation}
over all absolutely continuous mappings $x:[t_0,T]\rightarrow
\mathbb{R}^n$ satisfying the differential inclusion
\begin{equation}\label{May1}
\dot{x}(s)\in F(x(s)), \quad \mbox{ for  a.e. } s \in [t_0,T],
\end{equation}
with the initial condition
\begin{equation}\label{May2}
x(t_0)=x_0.
\end{equation}
Throughout the paper, $F:\mathbb{R}^n \rightrightarrows
\mathbb{R}^n$ is a given multifunction that satisfies a certain
set of  assumptions. In particular, the Hamiltonian
$H:\mathbb{R}^n\times \mathbb{R}^n \rightarrow \mathbb{R}$, defined
as
\begin{equation}\label{intro:H}
H(x,p)= \sup_{v \in F(x)} \langle v,p\rangle,
\end{equation}
is assumed to be semiconvex with respect to $x$ and differentiable
with respect to $p\neq 0$ with a locally Lipschitz gradient
$\nabla_p H(\cdot, p)$. Consider now an optimal trajectory
$\overline{x}$ of the Mayer problem and let $\overline{p}$ be any
dual arc associated with $\overline{x}$, that is,
$(\overline{x},\overline{p})$ is a solution of  the Hamiltonian
system
\begin{equation}\label{eq:char}
\begin{cases}
\hspace{.cm} \hspace{.3cm}\dot x(t)\in \partial_p^- H(x(t),p(t)) ,
\quad x(t_0)&=x_0,
\\
\hspace{.cm}
-\dot p(t)\in \partial_x^- H(x(t),p(t))
,\quad
-p(T) &\in \partial \phi(\overline{x}(T) )
\end{cases}
\quad \mbox{for a.e. } t\in[t_0,T],
\end{equation}
where $\partial_p^-H$ and $\partial_x^-H$ stand for the
subdifferentials of $H(x,\cdot)$  and $H(\cdot , p)$, respectively
\footnote{We  recall that $p(t) \neq 0$ for all $t \in [t_0,T]$
whenever $p(T)\neq 0$. In this case, the first inclusion in
\eqref{eq:char} becomes $\dot{x}(t)=\nabla_p H (x(t),p(t))$ for
all $t\in [t_0,T]$.}, and $\partial \phi$ for the generalized
gradient of $\phi$. In Section $3$, we show that
\begin{equation}\label{intro:sub}
-\overline{p}(t)\in \partial_{x}^{-} V(t,\overline{x}(t))\quad
\mbox{ for all } t \in [t_0,T],
\end{equation}
whenever the above inclusion is satisfied at $t_0$.
 This fact,
together with a similar result for the superdifferential obtained
in \cite{Nostro}, enables us to deduce that Fr\'echet
differentiability of $V(t,\cdot)$ propagates along optimal
trajectories whenever $\phi$ is differentiable (see
Proposition~\ref{Differentiability} below for a more general
statement). Furthermore, when $\phi$ is locally semiconcave, we
get the full first-order sensitivity relation \eqref{SR} on
$[t_0,T)$. Note that, if $V(t,\cdot)$ is differentiable at $\bar
x(t)$, then (\ref{intro:sub}) implies the well-known relation
$\bar p(t)=-\nabla _xV(t,\bar x(t))$. That is, (\ref{intro:sub})
links solutions of the so-called characteristic system
(\ref{eq:char}) to the gradient $\nabla_x V(t,\cdot)$ of the value
function along the trajectory $\bar x$.

If $V(t,\cdot)$ is twice Fr\'echet differentiable at $\bar x(t)$
for all $t \in [t_0,T]$ and $H$ is sufficiently smooth on a
neighborhood of $ \cup_{t \in [t_0,T] }(\bar x (t),\bar p (t))$,
then the Hessian $-\nabla^2_{xx} V(t,\bar x(t))$ is the solution
of the celebrated matrix Riccati equation:  $R(T)=- \nabla^2
\phi(\bar x(T))$ and
\begin{equation}\label{intro:RiccatiInc2}
\dot{R}(t)+ H_{px}[t]R(t)+R(t)H_{xp}[t] + R(t)H_{pp}[t]R(t) +
H_{xx}[t]=0,
\end{equation}
where $H_{px}[t]$ abbreviates $\nabla^2_{px} H
(\overline{x}(t),\overline{p}(t)),$ and similarly for $H_{xp}[t],
H_{pp}[t], H_{xx}[t]$.

Matrix Riccati equations have been intensively studied in connection with the  framework of the linear quadratic regulator problems and
have led to many applications. More recently, there has been interest
in their state dependent analogues, and extensions of  optimal
feedback expressions, based on solutions of a Riccati equation,  to
more general control problems have been sought.

Since, in general, the value function is not differentiable, in
what concerns  second-order sensitivity relations it is natural
to study  sub- and superjets. These sets provide  second-order upper and lower approximations of the value function in the
same way as the Fr\'echet sub- and superdifferentials do for first
order estimates (see Section $2.4$ below for the definition of superjets and
subjets). Such second-order sensitivity relations seem to be
absent from the current literature.

To derive a second-order analogue of \eqref{intro:sub}, set $q_0
:= \overline{p}(t_0)$, where $\overline{p}$ still denotes a dual
arc associated with $\overline{x}$, and let $Q_0$ be any symmetric
$n\times n$ matrix such that  $(-q_0,-Q_0)\in J^{2,-}_x
V(t_0,x_0)$, where the latter stands for the subjet of the
function $V(t_0,\cdot)$ evaluated at $x_0$. Assuming $H$ of class
$C^{2,1}(\mathbb{R}^n\times ( \mathbb{R}^n\smallsetminus \lbrace
0\rbrace ))$ and $\overline{p}\neq 0$, we consider
equation (\ref{intro:RiccatiInc2}) with the initial condition
$R(t_0)= - Q_0$.  Such a Riccati equation has a unique solution on its maximal
interval of existence which may escape to infinity in a finite
time because of the presence of a quadratic term. This is the main reason for
loosing differentiability of $V(t,\cdot)$, even when the cost
function $\phi$ and the Hamiltonian are smooth. So, let $a\in (t_0,T]$ be such that $R(\cdot)$ is well defined on the interval $[t_0,a]$. Then, we prove the following second-order
sensitivity relation for  the subjet of $V(t,\cdot)$ along
$\overline{x}$ :
\begin{equation}\label{intro:TesiSub}
(-\overline{p}(t),-R(t))\in J^{2,-}_x V(t,\overline{x}(t)) \quad \mbox{ for all } t\in [t_0,a].
\end{equation}
Moreover, when $\phi$ is locally semiconcave, we show that the
solution to \eqref{intro:RiccatiInc2} is global so that the above
inclusion is satisfied for every $t\in [t_0,T]$ (see Theorem
\ref{IncluRiccatiSub}). Finally, we complete our analysis of
sensitivity relations with a similar result for superjets, which
holds true backwards starting from time $T$ (see Theorem
\ref{Theo_SuperJet}).

In the last part of this paper, we use sensitivity analysis and comparison results for Riccati equations to investigate the second-order regularity of
 $V(t,\cdot)$ along an optimal trajectory $\overline{x}$.
Namely, suppose that $V(t_0,\cdot)$ is twice Fr\'echet
differentiable at $x_0$ and let $\bar x $ be an optimal trajectory
having a nonvanishing dual arc $\bar p$. Then we show that
$V(t,\cdot)$ stays twice Fr\'echet differentiable at
$\overline{x}(t)$ on the  interval of existence of  the solution
$R(\cdot)$ to  \eqref{intro:RiccatiInc2} with $R(t_0)=
-\nabla^2_{xx}V(t_0,x_0)$, and moreover
$R(t)=-\nabla^2_{xx}V(t,\overline{x}(t))$ (see Theorem
\ref{PropForward}). Furthermore, when $\phi$, hence $V$, is
locally semiconcave, such an interval
 extends to the whole set $[t_0,T]$. A similar statement holds true also ``backward" in time:  if
$\phi$   is twice Fr\'echet differentiable at $\bar x(T)$, then under some
technical assumptions (see Theorem \ref{Back_prop} for the details), $V(t,\cdot)$ is twice Fr\'echet differentiable at
$\bar x(t)$ for all $t$ such that the solution  $R(\cdot)$ of
(\ref{intro:RiccatiInc2}), with the final condition $R(T)=-\nabla
^2 \phi (\bar x(T))$, exists on $[t,T]$.

Finally, we investigate the local $C^2$ regularity of $V$ when the
cost $\phi$ is of class $C^2$.  For this aim we first guarantee
that the solution to the relevant Riccati equation exists globally
by proving the first-order sensitivity relation (\ref{intro:sub})
for  proximal---rather than Fr\'echet---subdifferentials. This allows us to deduce that if $V(t_0,\cdot)$
has a nonempty proximal subdifferential at $x_0$ and $\bar x$ is
an optimal trajectory starting at $x_0$ at time $t_0$, then $V(t,
\cdot)$ is of  class $C^2$ in a neighborhood of $\overline{x}(t)$
for all $t \in [t_0,T]$ (see Theorem
\ref{RegularityValueFunction}).  Let us recall that every lower
semicontinuos function defined on an open subset  $\Omega$ of a
finite dimensional space has a nonempty proximal subdifferential
on a dense subset of  $\Omega$. That is, in some sense,  our result
is about the generic local $C^2-$regularity of value functions
when data are sufficiently smooth. Earlier results in this
direction were obtained for the Bolza problem in
\cite{frankowska:hal-00851752,Cannarsa2013791} for $H$ strongly
convex with respect to $p$.  Since this is not the case of the
Hamiltonian associated to the Mayer problem, our results can not
be deduced from these earlier works and our proof is based on
different arguments.

This paper is organized as follows. In Section $2$, we fix the notation and we recall  basic material for later use.
 Section $3$ is devoted to first- and second-order sensitivity relations. Finally, Section $4$ concerns the propagation of
the  second-order Fr\'echet differentiability of $V(t,\cdot)$, as well as local $C^2$ regularity, along optimal trajectories.
\section{Preliminaries}
In this section, we review the basic concepts that are used in the sequel.
\subsection{Notation and basic facts}
Here, we quickly list the notation, various definitions, and basic facts. Further details can be found in several sources, for instance in
 \cite{MR1048347,MR2041617, MR709590,MR2662630}.

We denote by $|\cdot|$ the Euclidean norm in $\mathbb{R}^n$, by
$\langle \cdot,\cdot \rangle$ the inner product  and by  $[a, b]$
the segment connecting the points $a$ and $b$ of $\mathbb{R}^n$.
$B(x,\epsilon)$ and $\mathring{B}(x,\epsilon)$ are, respectively,
the closed and open balls of radius $\epsilon > 0$ centered at
$x$, and $S^{n-1}$ the unit sphere in $\mathbb{R}^n$. Moreover, we
use the shortened notation $B=B(0,1)$.  Define $a^+= \max\lbrace
a,0\rbrace $ for all $a\in\mathbb{R}$. For any subset $E$ of
$\mathbb{R}^n$, denote its boundary by $\partial E$ and its convex
hull by co$\,E$.

$M(n)$ is the set of $n\times n$ real matrices, $S(n)$ is the set
of symmetric $n\times n$ real matrices,  $\parallel Q \parallel$
denotes the
 operator norm and  $Q^*$ the transpose of $Q$ for any $Q\in M(n)$,    while $I_n$ is the $n\times n$ identity
 matrix. Recall that  $\parallel Q \parallel = \sup\{ |\langle Ax,x \rangle | : x \in S^{n-1} \}
 $  for every $
Q \in S(n) .$ \\
  For an \emph{extended real-valued function} $f:\mathbb{R}^n \rightarrow [-\infty,+\infty]$,
$dom (f):= \{ x\in \mathbb{R}^n:\ f(x)\neq \pm \infty\}$ is called
the domain of $f$. If $f : [t_0,t_1] \rightarrow \mathbb{R}^n$ is
continuous, define $\|f\|_{\infty} = \max_{t\in[t_0,t_1]} |f(t)|$.
$C([t_0,t_1];\mathbb{R}^n)$ and $W^{1,1}\left([t_0,t_1];
\mathbb{R}^n\right)$ are the spaces of all continuous and
absolutely continuous functions $x:[t_0,t_1] \rightarrow
\mathbb{R}^n$, respectively.
Moreover, we usually refer to an absolutely continuous function $x:[t_0,t_1] \rightarrow \mathbb{R}^n$ as an arc. The space $C^k(\Omega)$,
 where $\Omega$ is an open subset of $\mathbb{R}^n$, is the space of all functions that are continuously differentiable $k$ times on $\Omega$.
  The  H\"{o}lder space $C^{k,m}(\Omega)$ consists of those functions having continuous derivatives up to order $k$ and such that the $k-$th partial derivatives are
   H\" {o}lder continuous with exponent $m$, where $0<m\leq 1$.  Consider a locally Lipschitz function $f: \Omega \subset \mathbb{R}^n \rightarrow \mathbb{R}$,
    where $\Omega$ is an open set. The gradient of $f$ is $\nabla f(\cdot)$, which exists a.e. in $\Omega$. Moreover, if $f$ is twice Fr\'echet differentiable
    at some $x\in \Omega$, then  $\nabla^2 f(x)$ denotes the second-order Fr\'echet derivative of $f$ at $x$, also called Hessian of $f$ at $x$.
    Note that $\nabla^2 f(x)$ is a symmetric matrix (see, e.g., \cite{dieudonne2006foundations}).\\
Let $f:\Omega \rightarrow \mathbb{R}$ be any real-valued function
defined on an open set $\Omega\subset\mathbb{R}^n$.  For any
$x\in\Omega$, the sets
\[
\partial^- f(x)=\left\lbrace p\in\mathbb{R}^n : \liminf_{y \rightarrow x} \frac{f(y)-f(x)-\langle p, y - x\rangle  }{\mid y - x \mid}\geq 0 \right\rbrace,
\]
\[
\partial^+ f(x)=\left\lbrace p\in\mathbb{R}^n : \limsup_{y \rightarrow x} \frac{f(y)-f(x)-\langle p, y - x\rangle  }{\mid y - x \mid}\leq 0 \right\rbrace
\]
are the \emph{(Fr\'echet) subdifferential} and \emph{superdifferential} of $f$ at $x$, respectively. Furthermore, a vector $p\in \mathbb{R}^n$ is said to
 be a \emph{proximal subgradient} of $f$ at $x\in\Omega$ if there exist  $c,\;\rho \geq 0$
 such that
$$ f(y)-f(x)-\langle p, y-x \rangle \geq - c \vert y-x \vert^2,\quad \forall y \in B(x,\rho).$$
The set of all proximal subgradients of $f$ at $x$, denoted by
$\partial^{-,pr} f(x)$, is referred to as the \emph{proximal
subdifferential} of $f$ at $x$. Note that $\partial^{-,pr} f(x)$,
which may be empty, is a subset of $\partial^- f(x)$. Moreover, if
$f$ is of class $C^{1,1}$ in a neighborhood of $x$, then
$\partial^{-,pr} f(x)$ is the singleton $\lbrace \nabla f (x)
\rbrace$. If $f$ is Lipschitz, a vector $\zeta$ is a
\emph{reachable gradient} of $f$ at $x\in \Omega$ if there exists
a sequence $\lbrace x_j \rbrace \subset \Omega$ converging to $x$
such that $f$ is differentiable at $x_j$ for all $j\in
\mathbb{N}$, and $ \zeta = \lim_{ j \rightarrow \infty} \nabla
f(x_j).$ Let $\partial^* f(x)$ denote the set of all reachable
gradients of $f$ at $x$. The \emph{(Clarke) generalized gradient}
of $f$ at $x\in\Omega$, denoted by $\partial f(x)$, is the set
co$\, (\partial^* f(x))$. In fact, in the definition of $\partial
f(x)$, one can take $\lbrace x_j \rbrace \subset \Omega\setminus
\Omega_0$, for any set $\Omega_0$ of Lebesgue measure zero, cf.
\cite{MR709590}.

For a mapping $G : \mathbb{R}^n \times \mathbb{R}^m \rightarrow \mathbb{R}$, associating to each $x \in \mathbb{R}^n$ and $y\in \mathbb{R}^m$ a real number,
 we denote by $\nabla_x G$, $\nabla_y G$ the partial gradients (when they do exist); the partial generalized gradients will be
 denoted by $\partial_x G$, $\partial_y G$, and similarly for the partial Fr\'echet/proximal sub/superdifferentials. If $G$ is twice differentiable, then $\nabla_{xx}^2 G$, $\nabla^2 _{yy} G$, and $\nabla^2_{xy} G$ stand for its partial Hessians.\\
Finally, for an open set $\Omega\subset \mathbb{R}^n$, $f :
\Omega\rightarrow \mathbb{R} $ is \emph{semiconcave} if it is
continuous in $\Omega$ and there
 exists a constant $c$ such that
$ f(x + h) + f(x - h) - 2 f(x) \leq c | h |^2, $ for all $x, h\in
\mathbb{R}^n$ such that $[ x - h, x + h] \subset \Omega$. We say
that a function $f$ is semiconvex on $\Omega$ if and only if $-f$
is semiconcave on $\Omega$. We recall below some properties of
semiconcave functions (for further details see, for instance,
\cite{MR2041617}).
\begin{proposition}\label{aug16}
Let $\Omega\subset \mathbb{R}^n$ be  open, $f : \Omega \rightarrow
\mathbb{R}$ be a semiconcave function with semiconcavity constant
$c$, and let $x \in \Omega$. Then, $f$ is locally Lipschitz on
$\Omega$ and
\begin{enumerate}
\item $p\in \mathbb{R}^n$ belongs to $\partial^+ f(x)$ if and only
if, for any $y \in \Omega$ such that $[y, x] \subset \Omega$,
\begin{equation}\label{Booo}
f(y) - f(x) - \langle p, y - x \rangle  \leq c |y - x|^2.
\end{equation}
\item $\partial f(x) = \partial^{+} f(x)=co\ ( \partial^{\ast} f(x))$.
\item If $\partial^+ f(x)$ is a singleton, then $f$ is differentiable at $x$.
\end{enumerate}
\end{proposition}
If $f$ is semiconvex, then (\ref{Booo}) holds reversing the inequality and the sign of the quadratic term, and the other two statements are true with the subdifferential instead of the superdifferential. \\
The result below follows from \cite[ Proposition 1.1.3]{MR2041617} and from Theorems $2.3 $ and $2.8$ in
\cite{MR1742044}.
\begin{theorem}\label{Alex}
Let  $\Omega  \subset \mathbb{R}^n$ be open and  $f : \Omega \to
\mathbb{R}$ be a differentiable semiconcave function with
semiconcavity constant $c$. Then, $f$ is twice Fr\'echet
differentiable a.e. in $\Omega$ and $\nabla^2f(x) \leq  c I$ for
a.e. $x \in \Omega$, in the sense of quadratic forms.
\end{theorem}

For a better justification of the main theorem in Section $4.1$,
we recall next a technical result.
\begin{lemma}\label{LemmaTacnico}
Let $G:\mathbb{R}^{k} \rightarrow \mathbb{R}^{k}$, $G\in
C^1(\Omega)$ for some open set $\Omega\subset\mathbb{R}^k$. Denote
by $y(\cdot;y_0)$ the solution to
\begin{equation}\label{CauchyLemma}
\left\lbrace
\begin{array}{l}
\dot{y}(t)= G(y(t)) \mbox{ on } [0,T],\\
y(0)= y_0,
\end{array}
\right.
\end{equation}
and assume that for some $\overline{y}_0\in\mathbb{R}^k$,
$\overline{y}(\cdot):=y(\cdot;\overline{y}_0)$ is defined on
$[0,T]$ and takes values in $\Omega$. Consider the linear system:
\begin{equation}\label{VarLemma}
\left\lbrace
\begin{array}{l}
\dot{\psi}(t)= D G(\overline{y}(t))\psi(t) \mbox{ on } [0,T] ,\\
\psi(0)=  I_k  ,
\end{array}
\right.
\end{equation}
and denote its solution by $\psi,\; \psi:[0,T] \rightarrow M(k)$.
Then, for all $y_0$ in a neighborhood of $\overline{y}_0$, we have
\begin{equation}\label{FromVatEquation}
y(t;y_0)=\overline{y}(t)+ \psi(t)(y_0 - \overline{y}_0)+ o_t(\mid
y_0 - \overline{y}_0 \mid) \quad \mbox{ on } [0,T]
\end{equation}
where
$$
\lim_{y \rightarrow \overline{y}_0} \frac{o_t(\mid y-\overline{y}_0 \mid ) }{|y-\overline{y}_0| }=0,
$$
uniformly in $t$. That is, $\psi$ is the derivative of the map
$y_0 \mapsto y(\cdot;y_0)\in C([0,T];\mathbb{R}^n) $ evaluated at
$\overline{y}_0$.
\end{lemma}

\subsection{Second-order superjets and subjets}
We start by recalling the definition of the second-order superjets and subjets. For a comprehensive treatment and references to the literature on
 this subject we refer to \cite{Crandall,MR1118699}.
\begin{definition}
Let $f:\mathbb{R}^n \rightarrow [-\infty,+\infty]$ be an extended
real-valued function and let $x\in \textit{dom} (f)$. A pair
$(q,Q)\in \mathbb{R}^n\times S(n)$ is said to be a \emph{superjet}
of $f$ at $x$ if or some $\delta>0$ and for all $y\in
B(x,\delta)$,
\begin{equation}\label{SuperJet}
f(y)\leq f(x) + \langle q, y-x\rangle + \frac{1}{2} \langle
Q(y-x),y-x\rangle + o(\mid y-x \mid^2 ).
\end{equation}
The set of all the superjets of $f$ at $x$ is denoted by $J^{2,+}
f(x)$. Similarly, a pair $(q,Q)\in \mathbb{R}^n\times S(n)$ is
called a \emph{subjet} of $f$ at $x$ if there exists $\delta>0$
such that, for all $y\in B(x,\delta)$,
\begin{equation}\label{SubJet}
f(y)\geq f(x) + \langle q, y-x\rangle + \frac{1}{2} \langle
Q(y-x),y-x\rangle + o(\mid y-x \mid^2 ).
\end{equation}
The set of all the subjets of $f$ at $x$ is denoted by $J^{2,-}
f(x)$.
\end{definition}
Equivalently, one can define the set of all the subjets of
$f$ at $x$ as $ J^{2,-}f(x):= -J^{2,+}(-f(x)). $
\begin{remark}\label{aug16a}
By Proposition \ref{aug16}, if $f$ is semiconcave on a
neighborhood of $x$ with a semiconcavity constant $c$, then for
any $p \in \partial^+f(x)$  we have $(p,cI) \in J^{2,+} f(x).$
\end{remark}

From a geometrical point of view, the existence of a superjet
$(q,Q)\in J^{2,+} f(x)$ corresponds to the possibility of
``approximating'' $f$ from above by a $C^{2}$ function with the
gradient and Hessian given by $q$ and $Q$, respectively. This
leads to an equivalent definition of superjet in terms of the
so-called \emph{smooth test functions}.
\begin{proposition}\label{Equiv}
Let $f:\mathbb{R}^n \rightarrow [-\infty,+\infty]$ be an extended
real-valued function and let $x\in \textit{dom} (f)$. The
following statements are equivalent:
\begin{enumerate}
\item[(i)] $(q,Q)\in J^{2,+} f(x) $, \item[(ii)] there exists
$\phi\in C^2(\mathbb{R}^n;\mathbb{R})$ such that $f \leq \phi$,
$f(x)=\phi(x)$, and $(\nabla \phi(x), \nabla^2 \phi (x))=(q,Q)$.
\end{enumerate}
\end{proposition}
We give the proofs of this and of the next result below for the
readers convenience.
\begin{proof}
The nontrivial point of the conclusion is the existence of a
smooth function touching $f$ from above at $x$ with the gradient
and Hessian given by a fixed element in $J^{2,+}f(x)$. So, let
$(q,Q)\in  J^{2,+} f(x)$. Set $g(0)=0$, and for every $r>0$ define
\[
g(r):=\frac{1}{r^2} \sup_{y\in B(x,r)} \left( f(y)-f(x)-\langle
q,y-x \rangle -\frac{1}{2} \langle Q (y-x),y-x\rangle \right)^+ .
\]
By \eqref{SuperJet}, $g(r)$ tends to zero as $r\rightarrow 0^+$.
Let $\tilde{g}$ denote the upper envelope of $g$, defined
$\forall~r\geq 0$ as
\[ \tilde{g}(r)= \inf_{\delta > 0} \sup_{y\in [(r-\delta)^+, r+\delta]} g(y). \]
Then, $\tilde{g}$ is upper semicontinuous,
$\tilde{g}(r)\rightarrow 0$ when $r\rightarrow 0^+$, and
$\tilde{g}\geq g$. By a well-known result (see, e.g., Lemma
$3.1.8$ in \cite{MR2041617}), there exists a continuous
nondecreasing function $w:[0,+ \infty)\rightarrow [0,+\infty)$
such that $w(r)\rightarrow 0$ as $r\rightarrow 0^+$,
$\tilde{g}(r)\leq w(r)$ for any $r\geq 0$, and the function
$\gamma (r):=r w(r)$ is $C^1$ on $[0,+\infty)$, and satisfies
$\gamma'(0)=0$. Define for all $r\geq 0$,  $\beta(r)=\int_{r}^{2
r} \gamma(s)~ d s .$ From the relations  $
\beta(0)=\beta'(0)=\beta''(0)=0$ and $ \beta( r)\geq  r \gamma (r)
\geq r^2 \tilde{g}(r), $ we deduce that the function
\[
\phi(y):= f(x)+\langle q,y-x\rangle+\frac{1}{2}\langle
Q(y-x),y-x\rangle+ \beta(\mid y-x \mid),\quad y\in\mathbb{R}^n,
\]
belongs to $C^2(\mathbb{R}^n;\mathbb{R})$, and moreover $(\nabla
\phi(x), \nabla ^2\phi(x)) =( p, Q),\; \phi(x)=f(x)$, and
\[
\phi (y)-f(y)\geq \beta (\mid y-x \mid )-g(\mid y-x\mid )\mid y-x
\mid^2 \geq 0.
\]
This completes the proof.
\end{proof}
It is also clear that, in $(ii)$ above, one can replace the condition ``$\phi$ of class $C^2(\mathbb{R}^n;\mathbb{R})$'' by the
 condition ``$\phi$ of class $C^2$ in a neighborhood of $x$''. \begin{proposition}\label{Chiusura}
Let $f:\mathbb{R}^n \rightarrow [-\infty,+\infty]$ be an extended
real-valued function and let $x\in \textit{dom} (f)$. Then the
following properties hold:
\begin{itemize}
\item[(i)] $J^{2,+} f(x)$ is a convex subset of $\mathbb{R}^n
\times S(n)$, \item[(ii)] for any $q\in\mathbb{R}^n$, the set
$\lbrace Q\in S(n):~ (q,Q)\in J^{2,+} f(x)\rbrace$ is a closed convex
 subset of $S(n)$, \item[(iii)] if $f\leq g$ and
$f(\widehat{x})=g(\widehat{x})$ for some
$\widehat{x}\in\mathbb{R}^n$, then $J^{2,+} g(\widehat{x} )\subset
J^{2,+} f(\widehat{x})$. \item[(iv)] if $(q,Q)\in J^{2,+}f(x)$,
then $(q,Q')\in J^{2,+}f(x)$ for all $Q'\in S(n)$ such that
$Q'\geq Q$. Thus, the set $J^{2,+}f(x)$ is either empty or
unbounded.
\end{itemize}
\end{proposition}
\begin{proof} Below we prove only (ii). The other points are just as easy to prove. Note that convexity follows from (i).
 Suppose that the set $J^{2,+} f(x)$ is nonempty, and that for a fixed $q\in\mathbb{R}^n$ and $X_i\in S(n)$, $i=1,2,...$ satisfying
  $(q,X_i)\in J^{2,+} f(x)$ we have $X_i\rightarrow X$ as $i \rightarrow \infty$ for some $X\in S(n)$. Then, for any $\epsilon>0$ there exists $\delta_i$ such that for all $y\in B(x,\delta_i)$,
\[
f(y)\leq f(x) + \langle q, y-x\rangle + \frac{1}{2} \langle X_i
(y-x),y-x\rangle + \frac{ \epsilon}{2} \mid y-x \mid^2 .
\]
Now, choose $N$ such that $\parallel X_N -X\parallel \leq \epsilon$
and set $\delta:=\delta_N$. Then, for all $y\in B(x,\delta)$,
\[
f(y)- f(x) - \langle q, y-x\rangle - \frac{1}{2} \langle X_N
(y-x),y-x\rangle - \frac{1}{2} \langle (X- X_N) (y-x),y-x\rangle
\]
\[
\leq \frac{\epsilon}{2} \mid y-x\mid^2 +\frac{1}{2} \parallel X-
X_N \parallel \mid y-x\mid^2  \leq  \epsilon \mid y-x\mid^2.
\]
Thus, $(q,X)\in J^{2,+} f(x)$ and claim (ii) follows.
\end{proof}
\begin{remark}
The set $J^{2,+} f(x)$ is not necessarily closed. To see that,
consider the function $f : \mathbb{R}\rightarrow \mathbb{R}$
defined by $f(x)=0$ for $x\leq 0$ and $f(x)=-x$ for $x\geq 0$.
Then
$$J^{2,+} f(0)= \bigl( (-1,0)\times \mathbb{R}\bigr) \cup\bigl( \lbrace 0,-1\rbrace \times [0, + \infty)\bigr). $$
\end{remark}
One can get analogues of Propositions \ref{Equiv} and
\ref{Chiusura} with $J^{2,+}f(x)$ replaced by $J^{2,-}f(x)$.

Further remarks are in order. The sets of superjets or subjets
may be empty. However, if
$J^{2,+} f(x)$ and $J^{2,-} f(x)$ are both nonempty, then $f$ is
differentiable at $x$ and, for any $(q_1,Q_1)\in J^{2,+} f(x)$ and
$(q_2,Q_2)\in J^{2,-} f(x)$, we have that $q_1=q_2=\nabla f(x)$ and
$Q_1\geq Q_2$.

\subsection{Differential inclusions and main assumptions}
Here we introduce our main assumptions. We refer the reader to \cite{MR755330} and \cite{MR709590} for basic results on
differential inclusions, and to \cite{MR2728465} and  \cite{Nostro} for
a discussion of assumption $(H)$ below.

Consider a multifunction $F$ mapping $\mathbb{R}^n$ to the subsets of $\mathbb{R}^n$. Throughout this paper, we shall assume that $F$ satisfies the following collection of classical assumptions, which allows us to use the well-developed theory of differential inclusions:
$$\mbox{ \textbf{(SH)} }
\left\{\begin{array}{ll}
(i) & F(x) \mbox{ is nonempty, convex, compact for each } x \in \mathbb{R}^n,\\
(ii) & F \mbox{ is locally Lipschitz},\\
(iii)& \exists~ \gamma>0 \mbox{ so that } \max \lbrace \vert v
\vert : v\in F(x)\rbrace \leq \gamma (1+\vert x \vert)~ \forall x
\in \mathbb{R}^n.
\end{array}\right.$$
Recall that a multifunction $F:\mathbb{R}^n \rightrightarrows
\mathbb{R}^n$ with nonempty  compact values is \emph{locally
Lipschitz} if for each $x\in \mathbb{R}^n$ there exists a
neighborhood $K$ of $x$ and a constant $c>0$, which may depend on
$K$, such that $F(z)\subset F(y)+c \mid z-y \mid B$
for all $z, y \in K$.

We usually refer to the Mayer problem \eqref{Mayer}-\eqref{May2} as $\mathcal{P}(t_0,x_0)$. Under assumption $(SH)$, if $\phi$ is lower semicontinuous, then $\mathcal{P}(t_0,x_0)$ has at least one \emph{optimal solution},
 that is, a trajectory $\overline{x}(\cdot)\in W^{1,1}\left([t_0,T]; \mathbb{R}^n\right)$ of (\ref{May1})-(\ref{May2}) such that $\phi(\overline{x}(T))\leq \phi (x(T))$,
for any trajectory $x(\cdot)\in W^{1,1}\left([t_0,T]; \mathbb{R}^n\right)$ of (\ref{May1})-(\ref{May2}).

In this paper we impose, in addition to $(SH)$, the following assumptions on $F$ involving the Hamiltonian associated with $F$, that is, the function $H:\mathbb{R}^n\times \mathbb{R}^n \rightarrow \mathbb{R}$ defined by \eqref{intro:H}:
$$\mbox{ \textbf{(H)} }
\left\{
\begin{array}{ll}
&\mbox{for  every } r>0\\
&(i)\; \exists\ c \geq 0 \mbox{ so that },\; \forall p\in S^{n-1},\;  x \mapsto H(\cdot,p)\mbox{ is semiconvex on } B(0,r) \mbox{ with constant }  c, \\
&(ii)\; \nabla_p H(x,p) \mbox{ exists and is  Lipschitz in } x \mbox{ on } B(0,r), \mbox{ uniformly for }p \in \mathbb{R}^n \smallsetminus
\lbrace 0\rbrace.
\end{array}\right.$$
Lipschitz multifunctions with closed convex values always admit
parameterizations by Lipschitz functions (see e.g.
\cite{MR1048347}), but in general a smooth parameterizations does
not exist. We point out that, in \cite{MR2728465}, the authors
have provided a method to generate many examples of multifunctions
satisfying $(SH)$ and $(H)$ without, in general, admitting a
parameterizations being $C^1$ in $x$. Let us now quickly recall
some consequences of $(H)$. The semiconvexity of the map $x\mapsto
H(x,p)$ on $B(0,r)$ is equivalent to the mid-point property of the
multifunction $F$ on $B(0,r)$, that is,
\[
 2 F(x)\subset F(x+z)+F(x-z)+ c \mid z \mid^2 B
\]
for all $x, z$ such that $x\pm z\in B(0,r)$. Moreover, we have the following ``splitting'' formulas for $\partial H$.
\begin{lemma}\label{le:splitting}
Let $H$ be as in \eqref{intro:H}.
\begin{enumerate}
\item[(i)]  If  $H$ is locally Lipschitz and  $ H(\cdot,p)$ is
locally semiconvex, uniformly for $p \in S^{n-1}$, then
$$\partial H(x,p) \subset \partial_x^- H(x,p)\times \partial_p^- H(x,p), \quad \forall (x,p)\in\mathbb{R}^n\times\mathbb{R}^n.$$
\item[(ii)] Under assumptions (SH) and (H), we have that
\begin{equation}\label{eq:splitting}
\partial H(x,p) = \partial_x^- H(x,p)\times \partial_p^- H(x,p), \quad \forall (x,p)\in\mathbb{R}^n\times\mathbb{R}^n.
\end{equation}
\end{enumerate}
\end{lemma}
\begin{proof}
 Point $(i)$ is just \cite[Corollary~1]{MR2728465}. In order to prove $(ii)$,  it suffices to show that
 \begin{equation}\label{eq:semisplitting}
\partial_x^- H(x,p)\times \partial_p^- H(x,p)\subset \partial H(x,p), \quad \forall (x,p)\in\mathbb{R}^n\times\mathbb{R}^n.
\end{equation}
 Let $r>0$ and let $c_r\geq 0$ be a semiconvexity constant for   $H(\cdot,p)$ on $B(0,r)$, uniform for $p\in S^{n-1}$. By the convexity of $H(x,\cdot)$ and assumption $(H)$ we have that, for all $x,y\in B(0,r)$, all $p,q\in \mathbb{R}^n\setminus \{0\}$, and all
  $\xi\in\partial_x^-H(x,p)$,
 \begin{eqnarray*}
\lefteqn{H(y,q)-H(x,p)-\langle \xi,y-x\rangle -\langle \nabla_p H(x,p),q-p\rangle}
\\&=&H(y,q)-H(y,p)-\langle\nabla_p H(x,p),q-p\rangle+H(y,p)-H(x,p)-\langle \xi,y-x\rangle
\\
&\geq&
\langle\nabla_p H(y,p)-\nabla_p H(x,p),q-p\rangle-c_r\,| p|\, | y-x| ^2
\\
&\geq&
-k_r\,|y-x|\,|q-p|-c_r\,| p|\, | y-x| ^2,
\end{eqnarray*}
where $k_r$ is a Lipschitz constant for $\nabla_p H(\cdot,p)$ on $B(0,r)$, uniform for $p\in S^{n-1}$. Therefore, $(\xi, \nabla_p H(x,p))$ is a proximal subgradient of $H$ at $(x,p)$ and, as such, it belongs to $\partial H(x,p)$. This proves \eqref{eq:semisplitting} for $p\neq 0$.
 Thus, it remains to show that
 \begin{equation}\label{eq:0semisplitting}
\partial_x^- H(x,0)\times \partial_p^- H(x,0)\subset \partial H(x,0), \quad \forall x\in\mathbb{R}^n.
\end{equation}
Let $x\in \mathbb{R}^n$ and observe that $\partial_x^- H(x,0)=\{0\}$ because $H(x,0)\equiv 0$. So,
\begin{equation}\label{eq:*splitting}
\partial_x^- H(x,0)\times \partial_p^- H(x,0)=\mbox{co}\big(\{0\}\times \partial_p^*H(x,0)\big).
\end{equation}
Let $\eta\in
\partial_p^*H(x,0)$ and let $p_n\in \mathbb{R}^n\setminus \{0\}$
be such that $p_n\to 0$ and $\nabla_pH(x,p_n)\to\eta$ as
$n\to\infty$\footnote{Observe that this approximation is possible even when $H(x,\cdot)$ is differentiable at $0$.}. In view of \eqref{eq:splitting}, that we have
justified for $p\neq 0$, we have that $\partial H(x,p_n) =
\partial_x^- H(x,p_n)\times \{\nabla_p H(x,p_n)\}$ for all $n\geq
1$. Moreover, $\sup_{v \in \partial_x^- H(x,p_n)} |v|\to 0$ as
$n\to\infty$ because for some $k>0$,  $H(\cdot,p_n)$ is Lipschitz
on $B(x,1)$ with constant $k|p_n|$ for all $n\geq 1$. So, by the
upper semicontinuity of $\partial H$, $(0,\eta)\in\partial H(x,0)$
and \eqref{eq:0semisplitting} follows from \eqref{eq:*splitting}.
 This
completes the proof.
\end{proof}
Thanks to the above
properties, the maximum principle (see, e.g.,
\cite[Theorem~3.2.6]{MR709590} for a standard formulation) takes
the following form.
\begin{theorem}\label{TheoDualArc2}
Assume that (SH) and (H)(i) hold and  suppose $\phi:\mathbb{R}^n
\rightarrow \mathbb{R}$ is locally Lipschitz. If $x(\cdot)$ is an
optimal solution for $\mathcal{P}(t_0,x_0)$, then there exists an
arc $p:[t_0,T]\rightarrow \mathbb{R}^n$ which, together with
$x(\cdot)$, satisfies
\begin{equation}\label{HI2}
\left\{\begin{array}{rll}
\dot{x}(s) &\in& \partial_p^- H(x(s),p(s)), \\
-\dot{p}(s) &\in& \partial_x^- H(x(s),p(s)),
\end{array}\right.
\quad \mbox{for a.e. }\ s\in \left[ t_0, T \right]
\end{equation}
and
\begin{equation}\label{TC2}
-p(T)\in \partial \phi(x(T)).
\end{equation}
\end{theorem}
An absolutely continuous function $p(\cdot)$ satisfying the Hamiltonian system (\ref{HI2}) and the \emph{tranversality condition} (\ref{TC2}) is called a \emph{dual arc} associated with $x(\cdot)$. Moreover, if $v$ belongs to $\partial_p H(x,p)$, then $v\in F(x)$ and $\langle p, v \rangle = H(x,p)$. Thus, the system (\ref{HI2}) encodes the equality
\begin{equation}\label{MAXIMUMprincipleEQUATION}
H(x(t),p(t))= \langle p(t), \dot{x}(t) \rangle \mbox{ for a.e. } t \in [t_0,T].
\end{equation}
This equality shows that the scalar product $\langle v,p(t) \rangle$ is maximized over $F(x(t))$ by $v=\dot{x}(t)$, and for this reason the previous theorem is referred to as the \emph{maximum principle} (in Hamiltonian form).
\begin{remark}\label{RemarkDualArc} Let $\overline{p}$ be a dual arc associated with the optimal trajectory $\overline{x}$. Observe that under assumption $(SH)$ there are only two possible cases:
\begin{enumerate}
\item[(i)] either $\overline{p}(t)\neq 0$ for all $t\in [t_0,T]$,
\item[(ii)] or $\overline{p}(t)=0$ for all $t\in [t_0,T]$.
\end{enumerate}
Indeed, consider $r >0$ such that $\overline{x}([t_0,T]) \subset
B(0,r)$. If we denote by $c_r$ a Lipschitz constant for $F$ on
$B(0,r)$, then $c_r \vert p\vert$ is a Lipschitz constant for
$H(\cdot, p)$ on $B(0,r)$. Thus,
\begin{equation}\label{Impoooo}
\vert \zeta \vert \leq c_r \vert p \vert \quad \forall \zeta \in
\partial^-_x H(x,p),~ \forall x \in B(0,r),~ \forall p \in \mathbb{R}^n.
\end{equation}
Hence, in view of \eqref{HI2}, $ \vert
\dot{\overline{p}}(s)\vert\leq c_r \vert \overline{p}(s)\vert$ for
a.e. $s \in [t_0,T]$.
 Gronwall's Lemma allows to conclude.\\
\end{remark}
\begin{remark}\label{UltimoRemark}
Under the assumptions of Theorem \ref{TheoDualArc2}, if
$\partial^+\phi(\overline{x}(T))\neq \emptyset$, then for any $q\in
\partial^{+}\phi(\overline{x}(T))$ there exists an arc
$\overline{p}$ such that the pair $(\overline{x},\overline{p})$
satisfies \eqref{HI2} and $-\bar p(T)=q$. Indeed, since $q\in
\partial^{+}\phi(\overline{x}(T))$, there exists a function $g \in
C^1(\mathbb{R}^n;\mathbb{R})$ such that $g\geq \phi$,
$g(\overline{x}(T))=\phi(\overline{x}(T))$, and $ \nabla
g(\overline{x}(T))= q$ (see, for instance,
\cite[Proposition~3.1.7]{MR2041617}). Then, $\overline{x}$ is
still optimal for the Mayer problem \eqref{Mayer}-\eqref{May2}
with $\phi$ replaced by $g$. Thus, by Theorem \ref{TheoDualArc2}
there exists an arc $\overline{p}$ such that the pair
$(\overline{x},\overline{p})$ satisfies \eqref{HI2} and $-\bar
p(T)=q$.
\end{remark}
The existence of the gradient of $H$ with respect to $p$ at some $(x,p)$ is equivalent to the fact that the argmax set of $\langle v,p \rangle$ over $v\in F(x)$ is the singleton $\lbrace \nabla_p H(x,p)\rbrace$, that is,
\begin{equation}\label{derive}
H(x,p)= \langle \nabla_p H(x,p), p \rangle,\ \forall p \neq 0.
\end{equation}
In turn, \eqref{derive} implies that for every $x$ the boundary of $F(x)$ does not contain intervals $[a,b]$ with $a\neq b$. Moreover, assumption $(H)$ allows
 to state the following  result.
\begin{lemma}\label{LemmmaLip}
Assume that (SH) and (H) hold, and let $p:[t,T]\rightarrow \mathbb{R}^n\setminus \lbrace 0 \rbrace$ be continuous. Then for
each $\zeta \in \mathbb{R}^n$ the Cauchy problem
\begin{equation}\label{Cau}
\left\{\begin{array}{l}
\dot{y}(s)= \nabla_p H (y(s),p(s)) \quad \mbox{ for all } s\in \left[ t, T \right],\\
y(t)= \zeta,
\end{array}\right.
\end{equation}
has a unique solution $y(\cdot;t,\zeta )$. Moreover, for every  $
\zeta \in\mathbb{R}^n$ there exists $ k\geq 0$ such that $ |
y(s;t,\zeta )-y(s;t,z)|\leq e^{k (s-t)} |z-\zeta|$  for all $ z
\in B (\zeta,1)$,  $ s \in [t,T]. $
\end{lemma}
\begin{remark} \label{RemarkContinuity}
Suppose that $x(\cdot)$ is optimal for $\mathcal{P}(t_0,x_0)$ and $p(\cdot)$ is a nonvanishing dual arc associated to $x(\cdot)$. Then, Lemma \ref{LemmmaLip} implies that $x(\cdot)$ is the unique solution of (\ref{Cau}) with $t=t_0$ and $y(t_0)= x_0$. Furthermore, in this case, $x(\cdot)$ is of class $C^1$ and the maximum principle \eqref{MAXIMUMprincipleEQUATION} holds true for all $t\in [t_0,T]$. If $p\equiv 0$ is a dual arc, then \eqref{MAXIMUMprincipleEQUATION} continues to hold everywhere on $[t_0,T]$.
\end{remark}
The \emph{value function} $V:(-\infty,T] \times \mathbb{R}^n
\rightarrow \mathbb{R}$ associated to the Mayer problem  is
defined by: for all $(t_0,x_0)\in (-\infty,T]\times\mathbb{R}^n$
\begin{equation}\label{ValueFunction}
V(t_0,x_0)= \inf \left\lbrace \phi(x(T)) : x \in W^{1,1}\left([t_0,T]; \mathbb{R}^n\right) \mbox{ satisfies } \eqref{May1}-\eqref{May2} \right\rbrace.
\end{equation}
It is well known that, under assumption $(SH)$, if  $\phi$ is
locally Lipschitz, then $V$ is locally Lipschitz and satisfies, in
the viscosity sense, the Hamilton-Jacobi equation \eqref{HJB},
where $H$ is the Hamiltonian associated to $F$ by \eqref{intro:H}.
The semiconcavity of $V$ was investigated under the current
assumptions in \cite{MR2728465}. Finally, recall that $V$
satisfies the \emph{dynamic programming principle}. This means
that if $y(\cdot)$ is a trajectory of the system
\eqref{May1}-\eqref{May2}, then the function $s\mapsto
V(s,y(s))$ is nondecreasing, and  is constant if and only if
$y(\cdot)$ is optimal $\mathcal{P}(t_0,x_0)$.
\subsection{An overview of the theory of conjugate points in optimal control}
There is a close link between the theory of conjugate times, Riccati equations and second-order regularity properties of the value function. Here we wish to summarize some facts used hereafter. A detailed exposition of such results, more suited to the purposes of the present paper, is given in \cite{MR1344204, frankowska:hal-00851752}. Other details about Riccati equations can be found, e.g., in \cite{Reid:1253848}.\\

Consider  functions $\mathcal{H}:\mathbb{R}^n \times \mathbb{R}^n \rightarrow \mathbb{R}$, $\phi: \mathbb{R}^n \rightarrow \mathbb{R}$, and the Hamilton-Jacobi equation
\begin{equation}\label{HJBeqAncora}
\left\{\begin{array}{ll}
-\partial_t u(t,x) + \mathcal{H} (x, - u_x (t,x) )=0 &\mbox{in } (-\infty,T)\times \mathbb{R}^n,\\
u(T,x)=\phi(x), & x\in \mathbb{R}^n.
\end{array}\right.
\end{equation}
Under appropriate assumptions, \eqref{HJBeqAncora} has a unique viscosity solution which may be nonsmooth even when the data $\mathcal{H}$ and $\phi$ are smooth. On the other hand, the regularity of the solution of \eqref{HJBeqAncora} is strictly connected
 with the characteristic system below:
\begin{equation}\label{CJ*}
\left\{\begin{array}{rllrrl}
\dot{x}(t)&=& \nabla_p \mathcal{H} (x(t),p(t)), & x(T)&=&z, \\
-\dot{p}(t)&=& \nabla_x \mathcal{H} (x(t),p(t)), & -p(T)&=&\nabla \phi(z).
\end{array}\right.
\end{equation}
Assume that for some $z \in \mathbb{R}^n$ the solution of
(\ref{CJ*}) is defined on $(-\infty,T]$ and
 denote it by
$(x(\cdot,z),~ p(\cdot,z))$. Here we suppose that $\mathcal{H}$
and $\phi$ are of class $C^2$ in some neighborhoods of the set
$\bigcup_{t \leq T}\{(x(t,z),~ p(t,z))\}$ and the point $z$,
 respectively. If $\mathcal{H}$ is the Hamiltonian associated to the Mayer problem by \eqref{intro:H}, then \eqref{CJ*} coincides with the Hamiltonian system from Theorem \ref{TheoDualArc2}
  whenever $\nabla \phi (z)\neq 0$. By differentiating the solution map of \eqref{CJ*} with respect to $z$, we obtain that
  $\left(\frac{d}{d z} x(\cdot,z),~ \frac{d}{dz} p(\cdot,z)\right)$ satisfies the \emph{variational system}: for $t \in (-\infty , T ]$,
\begin{equation}\label{CP_prem}
\left\{\begin{array}{rllrrl}
\dot{X}(t)&=&\mathcal{H} _{xp} [t]X(t) + \mathcal{H}_{pp}[t] P(t) ,& X(T)&=& I, \\
-\dot{P}(t)&=& \mathcal{H} _{xx}[t] X(t)+ \mathcal{H} _{px} [t] P(t),&  -P(T)&=&\nabla^2\phi(z),
\end{array}\right.
\end{equation}
where $\mathcal{H}_{ij}[t]$ stands for $\nabla^2_{ij} \mathcal{H} (x(t,z),p(t,z))$, for $i,j\in\lbrace x,p\rbrace $. The solution of \eqref{CP_prem},
 which depends on $z$, is denoted by $(X,P)$. Note that $X(t)$ is invertible for $t$ sufficiently close to $T$, and moreover the function $R(t):=P(t)X(t)^{-1}$,
 as long as $X(t)$ is invertible, solves the Riccati equation
\begin{equation}\label{RiccatiGiusta}
\left\lbrace\begin{array}{l}
\dot{R}+ \mathcal{H}_{px}[t]R + R \mathcal{H}_{xp}[t]+R \mathcal{H}_{pp}[t] R+\mathcal{H}_{xx}[t]=0,\\
R(T)= - \nabla^2 \phi(z).
\end{array}
\right.
\end{equation}
Since $\mathcal{H}_{pp}[t], \;\mathcal{H}_{xx}[t]$, $\nabla^2
\phi(z)$ are symmetric matrices and
 $\mathcal{H}_{px}[t]=\mathcal{H}_{xp}[t]^*$, also  $R^*(\cdot)$  solves
equation \eqref{RiccatiGiusta}. From the uniqueness of the solution of
\eqref{RiccatiGiusta} we deduce that $R(\cdot)=R^*(\cdot)$, that
is, the values of $R(\cdot)$ are symmetric matrices.

The solution of \eqref{RiccatiGiusta} may escape to infinity in a finite time $t < T$, because of the presence of the quadratic term $R \mathcal{H}_{pp}[t] R$. For any $z \in \mathbb{R}^n$, define
\[
t_c=\inf \Big\lbrace t \in (-\infty,T]:~  R(s) \mbox{ is defined for all } s\in[t,T] \Big\rbrace.
\]
If $t_c > -\infty$, then it is called a \emph{conjugate time} for $z$, and $\lim_{t \searrow t_c} \parallel R(t) \parallel = + \infty$. Equivalently,
\[
t_c=\inf \Big\lbrace t\in (-\infty,T]:~ \det X(s)\neq 0 \mbox{ for all } s\in [t,T] \Big\rbrace,
\]
and, if $t_c>-\infty$, then $\det X(t_c)=0$. It is well known that,
if $\mathcal{H}_{pp}[\cdot]\geq 0$ on $[t_c,T]$, then the solution
of the linear equation
\begin{equation}\label{RiccatiGiusta2}
\left\lbrace\begin{array}{l}
\dot{Q}(t)+ \mathcal{H}_{px}[t]Q(t)+Q(t) \mathcal{H}_{xp}[t]+\mathcal{H}_{xx}[t]=0,\\
Q(T)= - \nabla^2 \phi(z).
\end{array}
\right.
\end{equation}
satisfies $Q(t)\leq R(t)$, in the sense that for any
$x\in\mathbb{R}^n$ we have that $\langle (R(t)-Q(t)) x , x \rangle
\geq 0$ for all $ t\in (t_c,T]$. In particular, for all $\eta \in
B$ and $t\in (t_c,T]$, $\langle R(t)\eta,\eta\rangle \geq -
\parallel Q(t) \parallel$. Since  $R(t)$ is symmetric,
if $t_c >-\infty$, then for some $x_t \in S^{n-1}$ we have
$\langle R(t) x_t , x_t \rangle \rightarrow +\infty$ as $t
\searrow t_c$.

If $(x(\cdot,z),~ p(\cdot,z))$  is given on a finite time interval
$[t_0,T]$, then the above definition of conjugate time can be
adapted, by saying that $t_c \in [t_0,T]$ is a conjugate time for
$z$ if $R(\cdot)$ is well defined on $(t_c,T]$ and $\lim_{t
\searrow t_c} \parallel R(t) \parallel = + \infty$. Equivalently,
if $\det X(t)\neq0$ for all  $t \in (t_c,T]$ and  $\det X(t_c)=0$.

In the present context, the occurrence of a conjugate time corresponds to the first time, proceeding backward, when $V(t,\cdot)$ stops to be twice differentiable at $x(t)$. More precisely, the assertion is as follows:
\begin{theorem}\label{TheoPntConju}
Let $(t_0,z_0)\in (-\infty,T]\times \mathbb{R}^n$ and assume that
a solution of  (\ref{CJ*}) with $z=z_0$ exists on $[t_0,T]$.
Define for $t\in [t_0,T]$, $\rho>0$
$$ M_{t,\rho}(z_0)= \lbrace \left( x(t),p(t)\right) \in \mathbb{R}^n \times \mathbb{R}^n :~ (x,p)\mbox{ solves } (\ref{CJ*})\mbox{ with } z\in \mathring{B}(z_0,\rho)  \rbrace.
$$
Assume that $\mathcal{H}$ is of class $C^2$ in a neighborhood of $\cup_{t\in[t_0,T]} M_{t,\delta}(z_0)$ for some $\delta>0$, and that $\phi$ is of class $C^2$ in a neighborhood of $z_0$. Consider any $\overline{t}\in [t_0,T]$. Then the following are equivalent:
\begin{enumerate}
\item there exists $\rho>0$ such that, for all $t \in [\overline{t},T]$, $M_{t,\rho}(z_0)$ is the graph of a continuously differentiable map $\Phi_t:\mathcal{O}_t \mapsto \mathbb{R}^n$, where $\mathcal{O}_t$ is an open subset of $\mathbb{R}^n$. Furthermore,
 the Jacobian $D \Phi_t$ of $\Phi_t$ satisfies
\begin{equation}\label{DerSec}
D \Phi_t (x(t))= P(t)X(t)^{-1} \quad \forall t \in [\overline{t},T],
\end{equation}
where $(x,p)$ and $(X,P)$ are the solutions of \eqref{CJ*} and
\eqref{CP_prem}  respectively with $z\in B(z_0,\rho)$; \item the
interval $[\overline{t},T]$ contains no conjugate time for $z_0$.
\end{enumerate}
Furthermore, suppose that $\mathcal{H}\in C^{2,1}$ in a
neighborhood of $\cup_{t\in[t_0,T]} M_{t,\delta}(z_0)$ for some
$\delta>0$, and, for some $m\in (0,1]$, $\phi\in C^{2,m}$ in a
neighborhood of $z_0$. If any of the conditions 1-2 holds true,
then there exist $\rho>0, \, c>0 $ such that $\Phi_t$ is of class
$C^{1,m}(\mathcal{O}_t)$ with H\" {o}lder constant $c$ for all $t
\in [\overline{t},T]$.
\end{theorem}
The equivalence of the above  two statements  was proved in
\cite{MR1344204}.  If $\phi\in C^{2,m}(\mathring{B}(z_0,\rho))$,
then one can easily deduce that $D \Phi_t (x(t))$ must be
H\"{o}lder continuous with exponent $m$, for all $t  \in  [\bar t,
T]$, thanks to \eqref{DerSec} and the fact that $(X,P)$ is the
solution of the linear system of ODEs \eqref{CP_prem}.
\begin{remark}
We point out that, when $\mathcal{H}$ is the Hamiltonian defined in
\eqref{intro:H} and $V$ is differentiable with respect to $x$, we
have already proved in \cite{Nostro} that $ \Phi_t
(x(t))=-\nabla_x V(t,x(t)) $  for all $t\in (t_c,T]. $ The first
point of  Theorem \ref{TheoPntConju} means that $V(t,\cdot)$ is of
class $C^2$ in a neighborhood of $x(t)$, and $ - \nabla^2_{xx}
V(t, x(t))= P(t)X(t)^{-1}= R(t).$ Moreover, if $\phi\in
C^{2,m}_{loc}$ for some $m\in (0,1]$ and
 $\mathcal{H}\in C^{2,1}$ in some suitable set, then $V(t,\cdot)$ is of class $C^{2,m}$ in a neighborhood of $x(t)$  for all $t \in (t_c,T]$.
\end{remark}
\section{Sensitivity relations}
This section is devoted to the analysis of first- and second-order
sensitivity relations.
\subsection{Subdifferentiability of the value function}
In this section, our first goal is to prove that the proximal subdifferentiability of $V(t,\cdot)$ propagates along optimal trajectories forward in time, as opposed to the proximal superdifferentiability that propagates backwards (see \cite[Theorem~4.1]{Nostro} for the result concerning propagation of the proximal superdifferential of $V$).
\begin{theorem}\label{Lemma_sub_prox}
Assume (SH), (H), and suppose $\phi: \mathbb{R}^n \rightarrow \mathbb{R}$ is locally Lipschitz. Let $\overline{x}:[t_0,T]\rightarrow \mathbb{R}^n$ be optimal for $\mathcal{P}(t_0,x_0)$ and let $\overline{p}:[t_0,T]\rightarrow \mathbb{R}^n$ be an
 arc such that $(\overline{x},\overline{p})$ is a solution of  the system
\begin{equation}\label{sub_prox}
\left\{\begin{array}{rllrrl}
\dot{x}(s) &\in& \partial_p^- H(x(s),p(s)), & x(t_0)&=&x_0 \\
-\dot{p}(s) &\in& \partial_x^- H(x(s),p(s)), & -p(t_0)&\in& \partial_{x}^{-,pr} V(t_0,x_0)
\end{array}\right.
\mbox{for a.e. }\ s\in \left[ t_0, T \right].
\end{equation}
Then, there exist constants $c, r>0$ such that, for all
$t\in[t_0,T]$ and $h \in B (0,r)$,
\begin{equation}\label{sub_prox_eq}
V(t,\overline{x}(t)+h)-V(t,\overline{x}(t)) \geq\  \langle -\overline{p}(t),h\rangle  - c \mid h \mid^2.
\end{equation}
Consequently,  $ -\overline{p}(t)\in
\partial_{x}^{-,pr}V(t,\overline{x}(t)) $ and $
(-\overline{p}(t),-c I_n )\in J^{2,-}_x V(t,\overline{x}(t)) $ for
all $ t\in[t_0,T].$
\end{theorem}
\begin{proof}
First recall that, by Remark~\ref{RemarkDualArc}, we can distinguish two cases:
\begin{enumerate}
\item[(i)] either $\overline{p}(t)\neq 0$ for all $t\in [t_0,T]$,
\item[(ii)] or $\overline{p}(t)=0$ for all $t\in [t_0,T]$.
\end{enumerate}
We shall analyze each of the above situations separately. Suppose first to be in case (i). Since  $-\overline{p}(t_0)\in \partial_{x}^{-,pr}V(t_0,x_0)$, there exist $c_0,r_0>0$ such that for every $h \in B(0,r_0)$,
\begin{equation}\label{Step1}
V(t_0,x_0+h)-V(t_0,x_0)  \geq - \langle \overline{p}(t_0), h \rangle -c_0 |h|^2.
\end{equation}
Fix $t\in( t_0,T]$. Recall that $\overline{x}(\cdot)$ is the unique solution of the final value problem
\begin{equation}\label{comealsolito}
\left\{
\begin{array}{l}
\dot{x}(s)= \nabla_p H(x(s),\overline{p}(s))\quad \mbox{ for all } s\in \left[ t_0, t \right],\\
x(t)=\overline{x}(t).
\end{array}\right.
\end{equation}
For all $h \in B$, let $x_h:[t_0,t] \rightarrow \mathbb{R}^n$ be
the solution of the system
\begin{equation*}\label{Vari}
\left\{
\begin{array}{l}
\dot{x}(s)= \nabla_p H(x(s),\overline{p}(s))\quad \mbox{ for all } s\in \left[ t_0, t \right],\\
x(t)=\overline{x}(t)+h.
\end{array}\right.
\end{equation*}
From  the optimality of $\overline{x}(\cdot)$, and the dynamic
programming principle we deduce that
\begin{equation}\label{Mah}
\begin{split}
V(t,\overline{x}(t)+h)-V(t,&\overline{x}(t))+\langle \overline{p}(t),h \rangle =
V(t,x_h(t))-V(t_0,x_0)+ \langle \overline{p}(t),h \rangle \\
&\geq V(t_0,x_h(t_0))-V(t_0,x_0)+ \langle \overline{p}(t),h \rangle .
\end{split}
\end{equation}
Observe that, since $F$ has a sublinear growth and $(H)(ii)$ holds
true, by a standard argument based on Gronwall's lemma it follows that for some constant $k$,
independent of $t\in (t_0,T]$,
\begin{equation}\label{StimaVecchia}
\| x_h- \overline{x} \|_{\infty}  \leq e^{k T} \mid h \mid, \quad
\forall~ h \in B .
\end{equation}
Setting $r := \min \lbrace 1, r_0 e^{-k T}\rbrace$, from
\eqref{Step1}, \eqref{Mah}, and \eqref{StimaVecchia} we deduce
that there exists a constant $c_1$ such that, for all $h\in
B(0,r)$,
\begin{equation}\label{Bohhhh}
V(t,\overline{x}(t)+h)-V(t,\overline{x}(t))+ \langle \overline{p}(t),h \rangle
\geq -\langle \overline{p}(t_0),x_h(t_0)-x_0\rangle +\langle \overline{p}(t),h \rangle- c_1 \mid h \mid^2.
\end{equation}
Now, we estimate the sum of the inner products in the right side of
\eqref{Bohhhh}. Thanks to (\ref{derive}) we get
\[
- \langle \overline{p}(t_0),x_h(t_0)- x_0\rangle+ \langle \overline{p}(t),h\rangle = \int_{t_0}^t \frac{d}{d s} \langle \overline{p}(s), x_h(s)-\overline{x}(s)\rangle\ ds
\]
\[
= \int_{t_0}^t  \langle \dot{\overline{p}}(s), x_h(s)-\overline{x}(s)\rangle + \langle \overline{p}(s), \dot{x}_h(s)-\dot{\overline{x}}(s)\rangle \ d s=
\]
\[
= \int_{t_0}^t  \langle \dot{\overline{p}}(s), x_h(s)- \overline{x}(s)\rangle + H(x_h(s),\overline{p}(s))-H(\overline{x}(s),\overline{p}(s)) \ d s.
\]
Since $-\dot{\overline{p}}(s)\in \partial_x^-
H(\overline{x}(s),\overline{p}(s))$ a.e. in $[t_0,T]$, assumption
(H)$\,(i)$ implies that
\begin{equation}\label{concluu}
- \langle \overline{p}(t_0),x_h(t_0)- x_0\rangle+ \langle \overline{p}(t),h\rangle \geq -c_2 \int_{t_0}^t  \mid \overline{p}(s)\mid \mid x_h(s)- \overline{x}(s)\mid^2,
\end{equation}
where $c_{2}$ is a suitable constant independent from $t \in [t_0,T]$. From \eqref{StimaVecchia} - \eqref{concluu} we conclude that there exists a constant $c$, independent of $t\in [t_0,T]$, such that \eqref{sub_prox_eq} holds true for all $h\in B(0,r)$. So, the claim is proved in case (i).

Next, suppose  to be in case (ii), that is $\overline{p}(t)=0$ for
all $t \in [t_0,T]$. Hence  $0\in
\partial_{x}^{-,pr}V(t_0,x_0)$ and there exist two constants
$r_0>0, c_0>0$ such that, for all $h\in B(0, r_0)$,
\begin{equation}\label{Step3}
V(t_0,x_0+h)-V(t_0,x_0) \geq  -c_0 |h|^2.
\end{equation}
Fix a time $t\in (t_0,T]$. By Filippov's Theorem (see, e.g.,  in \cite[Theorem 10.4.1]{MR1048347}), there exists $c_1>0$ independent
 of $t\in(t_0,T]$ such that, for any $h \in B$, the final value problem
\begin{equation}\label{lastAga2}
\left\lbrace
\begin{array}{l}
\dot{x}(s)\in F(x(s))\ \textit{ for a.e. } s\in[t_0,t],\\
x(t)=\overline{x}(t)+h.
\end{array}
\right.
\end{equation}
has a solution, $x_h(\cdot)$, that satisfies the inequality
\begin{equation}\label{Gr}
\parallel x_h - \overline{x} \parallel_{\infty} \leq c_1 \mid h \mid.
\end{equation}
Define  $\overline{r}=\min \lbrace 1, r_0/c_1 \rbrace$. By
\eqref{Step3} -  \eqref{Gr}, the dynamic programming principle,
and the optimality of $\overline{x}(\cdot)$ we deduce that for all
$h \in B(0,\overline{r})$
\begin{equation*}
V(t,\overline{x}(t)+h)- V(t,\overline{x}(t))\geq V(t_0,x_h(t_0))- V(t_0,x_0)\geq -c_0 c_1 ^2 \mid h\mid^2.
\end{equation*}
The proof is complete also in case (ii).
\end{proof}
The method of the above proof can be easily adapted to get a
similar conclusion for the Fr\'echet subdifferential of
$V(t,\cdot)$, as well.
\begin{theorem}\label{PropSubNonProx}
Assume (SH), (H), and suppose $\phi: \mathbb{R}^n \rightarrow
\mathbb{R}$ is locally Lipschitz. Let
$\overline{x}:[t_0,T]\rightarrow \mathbb{R}^n$ be optimal for
$\mathcal{P}(t_0,x_0)$ and let $\overline{p}:[t_0,T]\rightarrow
\mathbb{R}^n$ be an arc such that $(\overline{x},\overline{p})$ is
a solution of the system
\begin{equation}\label{sub_nonprox}
\left\{\begin{array}{rllrrl}
\dot{x}(s) &\in& \partial^-_p H(x(s),p(s)), & x(t_0)&=&x_0\\
-\dot{p}(s) &\in& \partial_x^- H(x(s),p(s)), & -p(t_0)&\in& \partial_{x}^{-} V(t_0,x_0)
\end{array}\right.
\mbox{for a.e. }\ s\in \left[ t_0, T \right].
\end{equation}
Then,
\begin{equation}\label{inclus_nonprox}
-\overline{p}(t)\in \partial_x^{-}V(t,\overline{x}(t))\quad \mbox{ for all }t \in [t_0,T].
\end{equation}
\end{theorem}

\begin{corollary}\label{Differentiability}
Assume (SH), (H) and suppose  $\phi: \mathbb{R}^n \rightarrow
\mathbb{R}$ is locally Lipschtiz. Let $\overline{x}(\cdot)$ be
optimal  for $\mathcal{P}(t_0,x_0)$, and suppose that
$\partial^+\phi(\overline{x}(T))\neq \emptyset$. Let
$\overline{p}(\cdot)$ be an arc such that
$(\overline{x},\overline{p})$ is a solution of
\begin{equation}\label{superHJ}
\left\{\begin{array}{rllrrl}
\dot{x}(s) &\in& \partial^-_p H(x(s),p(s)), & x(T)&=&\overline{x}(T)\\
-\dot{p}(s) &\in& \partial_x^- H(x(s),p(s)), & -p(T)&\in& \partial^+  \phi(\overline{x}(T))
\end{array}\right.
\mbox{for a.e. }\ s\in \left[ t_0, T \right].
\end{equation}
Then the following statements hold true:
\begin{itemize}
\item if $\partial^-_xV(t_0,x_0)  \neq \emptyset$, then
$V(t,\cdot)$ is differentiable at $\overline{x}(t)$ with $\nabla_x
V(t,\overline{x}(t))= -\overline{p}(t)$ for all $ t \in [t_0,T]$,
\item if $\phi$ is locally semiconcave and $\partial^-_xV(t_0,x_0)
\neq \emptyset$, then $V(\cdot,\cdot)$ is differentiable at
$(t,\overline{x}(t))$ with  $\nabla V(t,\overline{x}(t))=
\left(H(\overline{x}(t),\overline{p}(t)),-\overline{p}(t)\right)$
 for all $ t \in [t_0,T)$.
\end{itemize}
\end{corollary}
\begin{proof}
First, note that the existence of $\overline{p}$ such that
$(\overline{x},\overline{p})$ is a solution of \eqref{superHJ} is
guaranteed by
 Remark \ref{UltimoRemark}. From \cite[Theorem~3.4]{Nostro} we deduce that
\begin{equation}\label{super_nonprox}
-\overline{p}(t)\in \partial_{x}^{+}V(t,\overline{x}(t))\quad \mbox{ for all } t \in [t_0,T].
\end{equation}
Hence $V(t_0,\cdot)$ is differentiable at $x_0$. In particular,
$-\overline{p}(t_0)\in \partial^-_x V(t_0,x_0)$.
 Then, Theorem \ref{PropSubNonProx} ensures that \eqref{inclus_nonprox} holds true on $[t_0,T]$. The first statement of this corollary follows
 from \eqref{inclus_nonprox} together with \eqref{super_nonprox}.

The proof of the second one uses the same approach as \cite[proof
of Corollary~7.3.5]{MR2041617} and other standard arguments, so it
is omitted.
\end{proof}
Recall that the nondegeneracy condition $\partial^+\phi
(\overline{x}(T))\neq \emptyset$ is always satisfied when
$\phi$ is locally semiconcave or differentiable, and in both cases
$\overline{p}$ is a dual arc associated to $\overline{x}$. Let us
complete the analysis of  first-order sensitivity relations
with a result about  reachable gradients.
\begin{proposition}\label{TheoReach}
Assume (SH), (H) and suppose $\phi$ is locally Lipschitz and such
that $\partial^+\phi(z)\neq \emptyset$ for all $z\in\mathbb{R}^n$.
Then, for any  $(t,x)\in (-\infty,T]\times \mathbb{R}^n$ and
$\overline{q} \in
\partial^{*}_x V(t,x)\setminus \lbrace 0 \rbrace$, there exists a
solution  $(y(\cdot),p(\cdot))$ of
\begin{equation}\label{NewHamilt}
\left\lbrace
\begin{array}{rlll}
\dot{y}(s)& =  &\nabla_p H(y(s),p(s)),& y(t)=x \\
-\dot{p}(s)&\in &\partial_x^- H(y(s),p(s),& p(t)=-\overline{q}
\end{array}\right.
\mbox{for a.e. }s\in [t,T],
\end{equation}
such that $y(\cdot)$ is optimal for $\mathcal{P}(t,x)$ and the
following inclusion holds
\begin{equation}\label{reachable}
-p(s)\in \partial^{*}_x V(s,y(s))\quad \mbox{ for all }s\in [t,T].
\end{equation}
\end{proposition}
\begin{proof}
Take a sequence $x_k$ converging to $x$ such that $V(t,\cdot)$ is
differentiable at $x_k$ and $ \nabla_x V(t,x_k)$  converge to $
\overline{q}. $ Let $y_k(\cdot)$ be an optimal trajectory for
$\mathcal{P}(t,x_k)$. Then for every $k$ there exists a dual arc
$p_k$ such that the pair $(y_k,p_k)$ is a solution of the system
\begin{equation}\label{NewHamilt2}
\left\lbrace
\begin{array}{rlllrl}
\dot{x}_k(s)&\in &\partial_p^- H(x_k(s),p_k(s)), & \quad x_k(t)&=&x_k, \\
-\dot{p}_k(s)&\in &\partial_x^- H(x_k(s),p_k(s)), & \quad
-p_k(T)&\in& \partial^+ \phi (x_k(T))
\end{array}\right. \mbox{for  a.e.  } s \in [t,T].
\end{equation}
 By Corollary \ref{Differentiability}, $p_k(t)=-  \nabla_x V(t,x_k) $.  Hence, the sequence
$\{(y_k(\cdot),p_k (\cdot))\}$, after possibly passing to a
subsequence, converges uniformly to a pair of arcs
$(y(\cdot),p(\cdot))$ with $p$ nonvanishing on $[t,T]$.  Since the
set-valued map $\partial H$  is upper semicontinuous, we deduce
from Lemma \ref{le:splitting} that $(y(\cdot),p(\cdot))$ is, a
solution of \eqref{NewHamilt}. Moreover, $y$ is optimal for
$\mathcal{P}(t,x)$. In view of Corollary \ref{Differentiability},
for all $s\in[t,T]$
\begin{equation}\label{sisi}
p(s)=\lim_{k\rightarrow \infty} p_k(s)= - \lim_{k\rightarrow \infty} \nabla_x V(s,y_k(s)).
\end{equation}
Summarizing, for all $s\in [t,T]$ the sequence $y_k(s)$ converges
to $y(s)$ as $k\rightarrow \infty$, $V(s,\cdot)$ is differentiable
at $y_k(s)$ and $\nabla_x V(s,y_k(s)) \rightarrow p(s)$ as
$k\rightarrow \infty$. This ends the proof.
\end{proof}
\begin{remark}
If $\phi$ is locally semiconcave, then for any
$(\overline{q}_t,\overline{q})\in
\partial^* V(t,x) \backslash \{0\}$ there exists a pair $(y,p)$ such that $y$ is
optimal for $\mathcal{P}(t,x)$, and $(y,p)$ satisfies
\eqref{NewHamilt} - \eqref{reachable} on $(t,T]$, and for all
$s\in [t,T)$ it holds that
\begin{equation}\label{newCOnd}
\left( H(y(s),p(s)),-p(s) \right) \in \partial^* V(s,y(s)).
\end{equation}
Let us briefly justify \eqref{newCOnd}. Let $(t_k,x_k)$ be a sequence converging to $(t,x)$ such that $V$ is differentiable at $(t_k,x_k)$, and
 $\nabla V(t_k,x_k)$ converges to $(\overline{q}_t,\overline{q})$. If $y_k$ and $p_k$ are an optimal trajectory and an associated dual arc for the problem
 $\mathcal{P}(t_k,x_k)$, then, by Corollary \ref{Differentiability},  $\left(
H(y_k(s),p_k(s)), - p_k(s)\right)= \nabla V (s,y_k(s))$ for all
$s\in [t_k,T)$.  Taking a subsequence of $\{(y_k,p_k)\}$
converging to a solution $(y,p)$ of \eqref{NewHamilt}, we deduce
that \eqref{reachable}  holds true on $(t,T]$  and \eqref{newCOnd}
is satisfied for all $s\in (t,T)$. Moreover, since the reachable
set of $V$ is upper semicontinuous, \eqref{newCOnd} holds true on
$[t,T)$.
 \end{remark}

\subsection{Second-order sensitivity relations}
As recalled in Section $2.3$, it is well known that, when $H$ and
$\phi$ are sufficiently smooth, the twice differentiability of the
value function propagates along every optimal trajectory backward
in time as long as a conjugate time does not occur, and
$$(\nabla_x V(t,x(t)),\nabla_{xx}^2 V(t,x(t))=(-p(t),-R(t)),$$
 where $p$ is the dual arc associated to $x$ and $R$ is the solution of a Riccati equation.
  Here we prove that a pair $(-p,-R)$ propagates, as long as $R$ is well
  defined, in the set of all superjets or subjets of $V(t,\cdot)$ without assuming the twice differentiability of $V(t,\cdot)$. Let us start with  backward propagation of superjets.
To simplify the notation, for a fixed pair $(\overline{x}(\cdot),\overline{p}(\cdot))$ set, from now on,
 $H_{px}[t]:=\nabla^2_{px} H (\overline{x}(t),\overline{p}(t)),$
and let $H_{xp}[t], H_{pp}[t], H_{xx}[t]$ be defined analogously.
\begin{theorem}\label{Theo_SuperJet}
Assume (SH),  and suppose $H \in C^2(\mathbb{R}^n\times ( \mathbb{R}^n\smallsetminus \lbrace 0\rbrace ))$ and $\phi: \mathbb{R}^n \rightarrow \mathbb{R}$ is
 locally Lipschitz. Consider an optimal solution  $\overline{x}(\cdot)$ for  $\mathcal{P}(t_0,x_0)$ and suppose that there exists
   $(q,Q)\in J^{2,+}\phi(\overline{x}(T))$ with $ q\neq 0$. 
Let $\overline{p}(\cdot)$ be the dual arc associated with $\overline{x}(\cdot)$ such that $-\overline{p}(T)=q$, and let $R$ be the solution of the Riccati
equation
\begin{equation}\label{RiccatiInc}
\left\lbrace
\begin{array}{l}
\dot{R}(t)+ H_{px}[t]R(t)+R(t)H_{xp}[t] + R(t)H_{pp}[t]R(t) + H_{xx}[t]=0, \\
R(T)=-Q,
\end{array}
\right.
\end{equation}
defined on $[a,T]$ for some $a\in [t_0,T)$. Then the following second-order sensitivity relation holds true
\begin{equation}\label{Tesi}
(-\overline{p}(t),-R(t))\in J^{2,+}_x V(t,\overline{x}(t)) \quad \mbox{ for all } t\in [a,T].
\end{equation}
\end{theorem}
\begin{proof}
Since $(q,Q)\in J^{2,+}\phi(\overline{x}(T))$, by Proposition \ref{Equiv} there exists a function $g$ of class $C^2$ such that
$$ \phi \leq g,\quad g(\overline{x}(T))=\phi(\overline{x}(T)), \quad ( \nabla g(\overline{x}(T)),\nabla^2 g(\overline{x}(T)))=(q,Q).$$
Consider the Mayer problem of minimizing the final cost $g(x(T))$ over all trajectories $x(\cdot)$ of \eqref{May1}, \eqref{May2}, and its value function defined by
$$
W(t,z)=\inf \lbrace g(x(T)): x\mbox{ satisfies  \eqref{May1} a.e.
in }[t,T],~x(t)=z  \rbrace . $$ Since $\phi \leq g$, for all
$z\in\mathbb{R}^n$ and $t\in [t_0,T]$, $ W(t,z) \geq V(t,z). $
Moreover, since $g(\overline{x}(T))=\phi(\overline{x}(T))$, $
W(t,\overline{x}(t))=V(t,\overline{x}(t)). $ Thus the problem of
minimizing $g(y(T))$ over all admissible trajectories $y(\cdot)$
starting from $\overline{x}(t)$ at time $t$ has
$\overline{x}(\cdot)$ itself as optimal solution. Since $\nabla
g(\overline{x}(T))=q$,  $\overline{p}(\cdot)$ is the dual arc
associated with $\overline{x}(\cdot)$ for this new Mayer problem.
These facts, together with the equality $Q=\nabla^2
g(\overline{x}(T))$, imply that the Riccati equation linked to the
second-order derivative of $W(t,\cdot)$ at $\overline{x}(t)$ is
exactly \eqref{RiccatiInc}. At this stage, we appeal to Theorem
\ref{TheoPntConju} to deduce that $W(t,\cdot)$ is twice
continuously differentiable on a neighborhood of
$\overline{x}(t)$, for all $t\in [a,T]$, because there is no
conjugate point on such an interval and the Hessian $\nabla_{xx}^2
W(t,\overline{x}(t))$ is equal to $-R(t)$. It is also clear that
$\nabla_x W(t,\overline{x}(t))=-\overline{p}(t)$. Summarizing, we
have proved that for all $t\in[a,T]$ there exists a function
$W(t,\cdot)\geq V(t,\cdot)$,  of class $C^2$ on a neighborhood of
$\overline{x}(t)$, such that
$V(t,\overline{x}(t))=W(t,\overline{x}(t))$ and  $ (\nabla_x
W(t,\overline{x}(t)),\nabla^2_{xx} W(t,\overline{x}(t)))=
(-\overline{p}(t),-R(t))$.  Proposition \ref{Equiv} ends the
proof.
\end{proof}
\begin{remark}
From the $C^2$ regularity of $H$ on $\mathbb{R}^n\times (
\mathbb{R}^n\smallsetminus \lbrace 0\rbrace )$ and the form of
\eqref{RiccatiInc} it follows  that the values of  $R$ are
symmetric matrices. Note that the quadratic forms defined by
$R(t)$ cannot blow-up to $-\infty$ in finite time. Indeed, since
$H_{pp}[\cdot]\geq 0$ on $[t_0,T]$, we have $ -\dot{R}(t)\geq
H_{px}[t]R(t)+R(t)H_{xp}[t] + H_{xx}[t], $ which implies that
$R(t)\geq \tilde{R}(t)$ on the maximal interval of existence of
$R(\cdot)$, where $\tilde{R}$ solves
$$-\dot{\tilde{R}}(t)= H_{px}[t]\tilde{R}(t)+\tilde{R}(t)H_{xp}[t]
+ H_{xx}[t],\ \tilde{R}(T)=-Q.$$ Since $\tilde{R}(\cdot)$ does not
blow-up to $-\infty$ in finite time, the same is true for
$R(\cdot)$. On the other hand, it is well known that the
possibility of a blow-up to $+\infty$ on $[t_0,T]$ cannot be
excluded, in general. However, even though the propagation of
$(-\overline{p},- R)$ in the superjet of $V(t,\cdot)$ at
$\overline{x}(t)$ could stop at a possible blow-up time,
Theorem~3.1 in \cite{Nostro} provides sufficient conditions for
$J^{2,+}_x V(t,\overline{x}(t))$ to be nonempty for all
$t\in[t_0,T]$,
 because it implies that  for some $c \geq 0$,
$
(-\overline{p}(t),c I_n )\in J^{2,+}_x V(t,\overline{x}(t)) \quad \mbox{for all } t\in[t_0,T].
$
Note that, at the expense of assuming more smoothness of the Hamiltonian, the above theorem gives a significant improvement of Theorem~3.1 in \cite{Nostro}, before a possible blow-up time.
We finally recall that some sufficient conditions (more restrictive than the assumptions required here) for the existence of a global solution to \eqref{RiccatiInc} on $[t_0,T]$  were provided, for instance, in \cite{MR1223368}.
\end{remark}

It is natural to expect that the inclusion involving subjets of $V(t,\cdot)$ propagates forward along optimal trajectories.
\begin{theorem}\label{IncluRiccatiSub}
Assume (SH), and suppose $H \in C^{2,1}_{loc}(\mathbb{R}^n\times (
\mathbb{R}^n\smallsetminus \lbrace 0\rbrace ))$ and $\phi:
\mathbb{R}^n \rightarrow \mathbb{R}$ is
 locally Lipschitz. Consider an optimal solution  $\overline{x}(\cdot)$ for  $\mathcal{P}(t_0,x_0)$ and suppose that there exists a
  nonvanishing dual arc $\overline{p}(\cdot)$, associated with $\overline{x}(\cdot)$, such that  for some $R_0\in S(n)$
\begin{equation}\label{inSubJet}
(-\overline{p}(t_0),-R_0)\in J^{2,-}_x V(t_0,x_0).
\end{equation}
Let the solution $R$ of the Riccati equation
\begin{equation}\label{RiccatiInc2}
\left\lbrace
\begin{array}{l}
\dot{R}(t)+ H_{px}[t]R(t)+R(t)H_{xp}[t] + R(t)H_{pp}[t]R(t) + H_{xx}[t]=0,\\
R(t_0)=R_0,
\end{array}
\right.
\end{equation}
be defined on $[t_0,a]$ for some $a\in (t_0,T]$. Then
\begin{equation}\label{TesiSub}
(-\overline{p}(t),-R(t))\in J^{2,-}_x V(t,\overline{x}(t)) \quad \mbox{ for all } t\in [t_0,a].
\end{equation}
Moreover, if $\phi$ is locally semiconcave, then $a$ may be taken
equal to $T$.
\end{theorem}
\begin{proof} We first observe that values of $R(\cdot)$ are
symmetric matrices.  From \eqref{inSubJet} and Proposition
\ref{Equiv}, it follows that there exists a function $g \in C^2$
such that $V(t_0,\cdot)\geq g(\cdot)$, $g(x_0)=V(t_0,x_0)$, and
\begin{equation}\label{Condi}
(\nabla g(x_0),\nabla^2 g(x_0))=(-\overline{p}(t_0),-R_0).
\end{equation}
Consider the Hamiltonian system on $[t_0,T]$
\begin{equation}\label{Variation}
\left\{
\begin{array}{rlrl}
\dot{y}(t)&= \nabla_p H(y(t),q(t)),& y(t_0)&=y_0, \\
-\dot{q}(t)&= \nabla_x H(y(t),q(t)),& -q(t_0)&= \nabla g (y_0),
\end{array}\right.
\end{equation}
and denote its solution by $(y(\cdot;t_0,y_0), q(\cdot;t_0,y_0))$.
Then for some $\delta>0$, $H$ is of class $C^{2,1}$ in a
neighborhood $\mathcal{V}$ of $\cup_{t\in [t_0,T]}
(\overline{x}(t),\overline{p}(t) )$ and $(y(\cdot;t_0,y_0),
q(\cdot;t_0,y_0))$ takes value in $\mathcal{V}$ for any $y_0 \in
B(x_0 ,\delta)$. So, for every $y_0\in B(x_0 ,\delta)$, the
solution of \eqref{Variation} is $C^1$. Therefore, thanks to
\eqref{Condi} and the fact that the above system has a unique
solution, we deduce that $y(\cdot;t_0, x_0)= \overline{x}(\cdot)$,
and $q(\cdot; t_0 , x_0) =\overline{p}(\cdot)$. Fix $y_0\in
B(x_0,\delta)$. To simplify the notation, from now on denote the
 pair $(y(\cdot;t_0,y_0), q(\cdot;t_0,y_0)) $ by $(y(\cdot),q(\cdot))$.
Observe next that the mapping  $\tilde{R}(t):= R( T + t_0 - t )$
is the solution on $[T + t_0 - a,T]$ of the Riccati equation
\begin{equation}\label{TildeRicc}
\left\lbrace
\begin{array}{l}
\dot{\tilde{R}}(t)+ \tilde{H}_{px}[t]\tilde{R} (t)+\tilde{R} (t) \tilde{H}_{xp}[t] + \tilde{R} (t) \tilde{H}_{pp}[t]\tilde{R} (t) +  \tilde{H}_{xx}[t]=0,\\
\tilde{R}(T)= - \nabla^2 g (x_0 ),
\end{array}
\right.
\end{equation}
where $\tilde{H}_{px}[t]:= - H_{px}(\bar x(T + t_0 - t),\bar p (T
+ t_0 - t))$, and similarly for $\tilde{H}_{xx}[t],\,
\tilde{H}_{pp}[t], \, \tilde{H}_{xp}[t]$. Then, the interval $[T +
t_0 - a,T]$ contains no conjugate times of $\tilde{R}$ for $x_0$.
Thus, we appeal to Theorem \ref{TheoPntConju} to obtain, in terms
of $R$ and $(y,q)$, that for some $\rho \in (0,\delta)$ and all $t
\in [t_0,a]$, the set
$$M_{t,\rho}=\{(y(t;y_0),q(t;y_0)) \;:\; y_0 \in \mathring{B}(x_0,\rho)\}$$
is the graph of a continuously differentiable map $\Phi_t$ from an
open set $\mathcal{O}_t$ into $\mathbb{R}^n$.  Taking $\rho >0$
smaller, if needed,  and using again Theorem \ref{TheoPntConju},
we may assume that $\Phi_t$ are Lipschitz with a constant $k$
independent from $t \in [t_0,a]$. Hence, if $y_0$ is sufficiently
close to $x_0$,
\begin{equation}\label{22}
\mid q(t)-\overline{p}(t)\mid \leq k  \mid y(t)-\overline{x}(t)\mid,\mbox{ for all  } t \in[t_0,a],
\end{equation}
and, moreover,
\begin{equation}\label{Key}
R(t)(y(t)-\overline{x}(t))= q(t)-\overline{p}(t)+o_t(\mid
y(t)-\overline{x}(t) \mid ),
\end{equation}
where for every $t$, $\lim_{r \to 0+} \frac{o_t(r)}{r}=0.$ Then,
for some $k_1 \geq 0$, $o_t(\mid y(t)-\overline{x}(t) \mid )\leq
k_1 \mid y(t)-\overline{x}(t) \mid$ for all $t \in [t_0,a]$.
Consider, for every $t\in[t_0,a]$, the continuous  map
\begin{equation}\label{diff}
F_t: \mathring{B}(x_0,\rho) \ni z \longmapsto y(t;t_0,z).
\end{equation}
Since $M_{t,\rho}$ is the graph of a function,  $F_t$ is
injective.  Consequently, by the invariance-of-domain theorem,
$F_t( \mathring{B}(x_0,\rho))$ is  open.

Fix  $t \in [t_0,a]$. Now, we want to show that, for every $y_0\in
B(x_0,\delta)$,
\begin{equation}\label{Tay2}
\begin{split}
V(t,y(t))&-V(t,\overline{x}(t))\geq \langle -\overline{p}(t),y(t)-\overline{x}(t) \rangle \\
&-\frac{1}{2}\langle R(t)(y(t)-\overline{x}(t)), y(t)-\overline{x} (t) \rangle + o (\mid y(t)-\overline{x}(t) \mid^2).
\end{split}
\end{equation}
The conclusion \eqref{TesiSub} will follow from \eqref{Tay2}
because $F_t( \mathring{B}(x_0,\rho))$ is an open neighborhood of
$\bar x(t)$. So, let us prove \eqref{Tay2}. From the dynamic
programming principle and \eqref{inSubJet}, we deduce that, for all $t\in[t_0,a]$,
\begin{equation}\label{Prima}
\begin{split}
V(t,y(t))&-V(t,\overline{x}(t))\geq V(t_0,y(t_0))-V(t_0,\overline{x}(t_0))\geq
\langle -\overline{p}(t_0),y(t_0)-\overline{x}(t_0) \rangle \\
&-\frac{1}{2}\langle R(t_0)(y(t_0)-\overline{x}(t_0)), y(t_0)-\overline{x} (t_0) \rangle + o(\mid y(t_0)-\overline{x}(t_0) \mid^2)\\
=\langle - \overline{p}(t),&y(t)-\overline{x}(t) \rangle  -\frac{1}{2}\langle R(t)(y(t)-\overline{x}(t)), y(t)-\overline{x}(t) \rangle + o(\mid y(t_0)-x(t_0) \mid^2)\\
&+ \int_{t_0}^t \frac{d}{d s} \Big( \langle\overline{p}(s),y(s)-\overline{x}(s) \rangle + \frac{1}{2} \langle R(s)(y(s)-\overline{x}(s)),y(s)-\overline{x}(s)\rangle\Big) ~ d s.
\end{split}
\end{equation}
We shall evaluate the integral in the right side of \eqref{Prima}.
The symmetry of the values of $R(\cdot)$ yields
\begin{equation*}
\begin{split}
&\int_{t_0}^t \frac{d}{d s} \Big( \langle \overline{p} (s),y(s)-\overline{x}(s) \rangle + \frac{1}{2} \langle R(s)(y(s)-\overline{x}(s)),y(s)-\overline{x}(s)\rangle\Big) ~ d s \\
&= \int_{t_0}^t \Bigg( \langle \overline{p}(s),\dot{ y}(s)-\dot{\overline{x}}(s) \rangle + \langle \dot{\overline{p}}(s),y(s)-\overline{x}(s) \rangle\Bigg)~ ds \\
&+ \int_{t_0}^t \Bigg( \langle R (s)(y(s)-\overline{x}(s)),\dot{y}(s)-\dot{\overline{x}} (s)\rangle +   \frac{1}{2} \langle \dot{R}(s)(y(s)-\overline{x}(s)),y(s)-\overline{x}(s)\rangle \Bigg)~ ds \\
&=(I) + (II).
\end{split}
\end{equation*}
For the first term we use \eqref{Variation}:
\begin{equation}\label{Lungo1}
(I)= \int_{t_0}^t \Bigg( \langle \overline{p}(s),\nabla_p H(y(s),q(s))- \nabla_p H(\overline{x}(s),\overline{p}(s)) \rangle  - \langle \nabla_x H(\overline{x}(s),\overline{p}(s)),y(s)-\overline{x}(s) \rangle    \Bigg) ~ ds.
\end{equation}
For the second one, we first use \eqref{RiccatiInc2},
\eqref{Variation}, together with the symmetry of the values of
$R(\cdot)$:
\begin{equation}\label{Lungo2}
\begin{split}
(II) = \int_{t_0}^t &\Bigg(
\langle R (s)(y(s)-\overline{x}(s)),\nabla_p H(y(s),q(s)) -\nabla_p H(\overline{x}(s),\overline{p}(s)) \rangle \\
&-\langle H_{px}[s]R(s)(y(s)-\overline{x}(s)),y(s)-\overline{x}(s)\rangle
-\frac{1}{2}  \langle H_{pp}[s]R(s)(y(s)-\overline{x}(s)),R(s)( y(s)-\overline{x}(s))\rangle \\
&-\frac{1}{2} \langle H_{xx}[s](y(s)-\overline{x}(s)),y(s)-\overline{x}(s)\rangle \Bigg) ~ ds.
\end{split}
\end{equation}
We appeal now to \eqref{22} and \eqref{Key} to deduce that
\begin{equation*}
\begin{split}
(II)
= \int_{t_0}^t \Bigg( &\langle q(s)-\overline{p}(s),\nabla_p H(y(s),q(s)) -\nabla_p H(\overline{x}(s),\overline{p}(s)) \rangle
-\langle H_{px}[s](q(s)-\overline{p}(s)),y(s)-\overline{x}(s)\rangle \\
&-\frac{1}{2}  \langle H_{pp}[s](q(s)-\overline{p}(s)),q(s)-\overline{p}(s)\rangle -\frac{1}{2} \langle H_{xx}[s](y(s)-\overline{x}(s)),y(s)-\overline{x}(s)\rangle\\
&+ o_s (\mid y(s)-\overline{x}(s)) \mid^2) \Bigg)~ ds,
\end{split}
\end{equation*}
where $o_s (\mid y(s)-\overline{x}(s)\mid^2 )\leq c_0 \mid
y(s)-\overline{x}(s)\mid^2 $ for some $c_0 \geq 0$, and  $\lim_{r
\searrow 0} \frac{o_s(r)}{ r} =0$ for all $s\in [t_0,t]$. We add
up $(I)$ and $(II)$. A short calculation shows that
\begin{equation}\label{PrimaOpiccolo}
\begin{split}
\int_{t_0}^t \frac{d}{d s} &\Big( \langle p(s),y(s)-\overline{x}(s) \rangle + \frac{1}{2} \langle R(s)(y(s)-\overline{x}(s)),y(s)-\overline{x}(s)\rangle\Big) ~ d s \\
= \int_{t_0}^t &\Bigg(  H(y(s),q(s))-H(\overline{x}(s),\overline{p}(s)) -\langle \nabla_x H(\overline{x}(s),\overline{p}(s)), y(s)-\overline{x}(s)\rangle \\
& - \langle \nabla_p H(\overline{x}(s),\overline{p}(s)), q(s)- \overline{p}(s)\rangle -\langle H_{px}[s](q(s)-\overline{p}(s)),y(s)-\overline{x}(s)\rangle\\
&-\frac{1}{2}\langle H_{pp}[s](q(s)-\overline{p}(s)),q(s)-\overline{p}(s)\rangle\\
&-\frac{1}{2}\langle H_{xx}[s](y(s)-\overline{x}(s)),y(s)-\overline{x}(s)\rangle  + o_s (\mid y(s)-\overline{x}(s)) \mid^2) \Bigg)~ ds .
\end{split}
\end{equation}
Since $H$ is of class $C_{loc}^{2,1}(\mathbb{R}^n\times (
\mathbb{R}^n\smallsetminus \lbrace 0\rbrace ))$, using  Taylor's
expansion we obtain that  the above integrand can
be estimated by $ \tilde{o}_s(\mid y(s)-\overline{x}(s)) \mid^2)$,
where  $\tilde{o}_s (\mid y(s)-\overline{x}(s)\mid^2 )\leq c_1
\mid y(s)-\overline{x}(s)\mid^2 $ for some $c_1 \geq 0$, and
$\lim_{r \searrow 0} \frac{\tilde{o}_s(r)}{ r} =0$ for all $s\in
[t_0,T]$. Recalling \eqref{Prima} yields
\begin{equation}\label{QuasiFine}
\begin{split}
V(t,y(t))&-V(t,\overline{x}(t))\geq \langle - \overline{p}(t),y(t)-\overline{x}(t) \rangle
-\frac{1}{2}\langle R(t)(y(t)-\overline{x}(t)), y(t)-\overline{x}(t) \rangle \\
&+ o(\mid y(t_0)-x_0 \mid^2)+ \int_{t_0}^t \tilde{o}_s(\mid y(s)-\overline{x}(s) \mid^2) ~ d s.
\end{split}
\end{equation}
We claim that  there exists $c>0$ such that
\begin{equation}\label{ultimi}
\mid y(s)-\overline{x}(s) \mid \leq c \mid y(t)-\overline{x}(t)\mid, \mbox{ for all } s\in [t_0,t].
\end{equation}
Indeed, from the local Lipschitz continuity of $\nabla
H(\cdot,\cdot)$, working with the system \eqref{Variation}, we
prove existence of a constant $k_1$ such that, for all
$s\in[t_0,t]$,
\begin{equation}\label{11}
 \mid y(s)-\overline{x}(s) \mid \leq k_1 \left( \mid y(t)-\overline{x}(t)\mid +  \mid q(t)-\overline{p}(t)\mid \right).
\end{equation}
By \eqref{22} and \eqref{11}  there exists a constant $c>0$ such
that \eqref{ultimi} holds true. Hence
$$
\lim_{y(t) \rightarrow \overline{x}(t)} \frac{\tilde{o}_s (\mid y (s) -\overline{x}(s) \mid^2 ) }{|y(t)-\overline{x}(t)\mid^2 }=0.
$$
From the Lebesgue dominated convergence theorem it follows that the integral in \eqref{QuasiFine} can be bounded by an infinitesimal of the form $o (\mid y (t) -\overline{x}(t) \mid^2 )$. This observation together with \eqref{QuasiFine} allows to conclude that, for every $y_0\in B(x_0,\delta)$ and $t\in[t_0,a]$,  inequality \eqref{Tay2} holds true. This completes the proof of \eqref{TesiSub}.

Finally, let us investigate the case when $\phi$ (hence of $V$, see
\cite{MR2728465}) is locally semiconcave. Set
$$\overline{t}=\sup\lbrace t\in [t_0,T]:\; R(\cdot) \mbox{ is defined on } [t_0,t] \rbrace.$$
Since $H_{pp}[t]\geq 0$ on $[t_0,T]$, we have that $\dot{R}(t)\leq
- H_{px}[t]R(t)- R(t)H_{xp}[t] - H_{xx}[t]$.
 Then, since the
solution of the linear equation $\dot{\tilde{R}} +
H_{px}[t]\tilde{R}(t)+ \tilde{R}(t)H_{xp}[t] + H_{xx}[t]=0, \
\tilde{R}(t_0)=R_0$ is well-defined on $[t_0,T]$, a constant
$c_1>0$ exists such that $R(t) \leq c_1 I$,  for any $t\in
[t_0,\overline{t})$. On the other hand, since $V$ is locally
semiconcave, by Remark \ref{aug16a}, for some $c_2>0$ and all $t
\in [t_0,T]$, $(-\bar p(t), c_2I) \in J^{2,+}_x V(t,\cdot)$. Hence
$-R(t) \leq c_2I$ on $[t_0,\bar t)$ and therefore $\|R(t)\| \leq
\max\{c_1,c_2\}$ on $[t_0,\bar t)$. So we can guarantee that $R$
can be extended to the maximal interval of existence $[t_0,T]$.
\end{proof}
\section{Second-order regularity of the value function along optimal trajectories}
In this section, we derive several results concerning propagation of the second-order regularity of the value function along optimal trajectories.
\subsection{Forward propagation of  twice Fr\'echet differentiability}
Our aim here is to prove that  twice Fr\'echet differentiability
of $V(t,\cdot)$ propagates along optimal trajectories forward in
time, as the first-order differentiability (see Corollary
\ref{Differentiability}). Note that we cannot ensure that such a
result holds true on the whole interval $[t_0,T]$, because of the
existence of possible blow-up times for the Riccati equation
below.
\begin{theorem}\label{PropForward}
Assume $(SH)$, and suppose  $H \in C^{2} (\mathbb{R}^n \times ( \mathbb{R}^n
\setminus \lbrace 0\rbrace))$ and $\phi$ is locally
Lipschitz with  $\partial^+ \phi(z)\neq \emptyset$ for all $z\in
\mathbb{R}^n$. Let $\overline{x}$ be an optimal solution for
$\mathcal{P}(t_0,x_0)$ such that
 $\partial^+ \phi(\overline{x}(T)) \neq \lbrace 0 \rbrace$ and consider a dual arc $\overline{p}$ satisfying
 $ -\overline{p}(T)\in \partial^+ \phi(\overline{x}(T))\setminus \lbrace 0 \rbrace $.
If $V(t_0,\cdot)$ is continuously differentiable on a neighborhood
of $x_0$, then $V(t,\cdot)$ is continuously differentiable on a
neighborhood of $\bar x(t)$ for all $t \in [t_0,T]$.
Furthermore,  if $V(t_0,\cdot)$ is  twice Fr\'echet differentiable at
$x_0$ and the solution $R$ of
\begin{equation}\label{DeivSecBase}
\left\lbrace
\begin{array}{l}
\dot{R}(t)+ H_{px}[t]R(t)+R(t)H_{xp}[t] + R(t)H_{pp}[t]R(t) + H_{xx}[t]=0, \\
R(t_0)=-\nabla_{xx}V(t_0,x_0),
\end{array}
\right.
\end{equation}
is defined on $[t_0,a]$ for some $a\in ( t_0, T]$, then
$V(t,\cdot)$ is  twice Fr\'echet differentiable at
$\overline{x}(t)$ for any $t\in [t_0,a]$. Moreover, $R(t)=-
\nabla_{xx} V(t,\overline{x}(t))$.
\end{theorem}
\begin{proof} Corollary \ref{Differentiability} and Remark
\ref{RemarkDualArc} yield the existence of  $\rho >0$  such that
$V(t_0,\cdot)$ is continuously differentiable on $B(x_0,\rho)$ and
$\nabla_xV(t_0,y_0)\neq 0$ for all $y_0 \in B(x_0,\rho)$.  Fix an
optimal solution $x$ of $\mathcal{P}(t_0,y_0)$, where $y_0 \in
B(x_0,\rho)$. From Corollary \ref{Differentiability} and Remark
\ref{RemarkDualArc} we deduce that $p(t_0)= -\nabla_xV(t_0,y_0)$
and $p$ never vanishes. Consequently $(x,p)$ solves the
Hamiltonian system
\begin{equation}\label{Variation1}
\left\{
\begin{array}{rlrl}
\dot{x}(t)&= \nabla_p H(x(t),p(t)),& x(t_0)&=y_0, \\
-\dot{p}(t)&= \nabla_x H(x(t),p(t)),& p(t_0)&=
-\nabla_xV(t_0,y_0).
\end{array}\right.  t \in [t_0,T]
\end{equation}
Since the solution of the above system is unique, we deduce that
the optimal solution and the corresponding dual arc are unique for
every $y_0 \in B(x_0,\rho)$. By Corollary \ref{Differentiability},
$V(t,\cdot)$ is differentiable at $x(t)$ and
$p(t)=-\nabla_xV(t,x(t))$  for all $t \in [t_0,T]$. Denote the
solution of (\ref{Variation1}) by $(x(\cdot;y_0),p(\cdot;y_0))$,
or, when the initial point can be omitted, by $(x,p)$.

The map $\Gamma_t$ defined by $\mathring{B}(x_0,\rho)\ni z \mapsto
x(t;z)=:\Gamma_t(z)$ is continuous.  We claim that it is also
injective. Indeed if $x(t,z_1)=x(t,z_2)$  for some $z_1, \, z_2
\in \mathring{B}(x_0,\rho)$, then $p(t,z_1)=
-\nabla_xV(t,x(t,z_1))=-\nabla_xV(t,x(t,z_2))=p(t,z_2)$.  We
deduce that $z_1=z_2$ from the uniqueness of solutions to ODEs with
locally Lipschitz right-hand side and a fixed final condition. By
the invariance-of-domain theorem, $\Gamma_t$ is a homeomorphism
from $\mathring{B}(x_0,\rho)$ onto the open set
$\Gamma_t(\mathring{B}(x_0,\rho))$. Using the fact that solutions
of (\ref{Variation1}) depend on the initial conditions in a
continuous way, we deduce that $\nabla_xV(t,\cdot)$ is continuous
on $\Gamma_t(\mathring{B}(x_0,\rho))$.  This implies the first
statement of theorem.

To prove the second one, we start with an observation. Let
$(Y,Q):[t_0,T]\rightarrow M(n)\times M(n)$ be the solution of the
system
\begin{equation}\label{LineForw}
\left( \begin{array}{c} \dot{Y} \\ \dot{Q} \end{array} \right) =
\begin{bmatrix}\; H_{xp}[s] & \; H_{pp}[s] \\ - H_{xx}[s] & -
H_{px}[s] \end{bmatrix} \left( \begin{array}{c}  Y \\ Q\end{array}
\right),\quad \left( \begin{array}{c}  Y(t_0) \\ Q(t_0)
\end{array} \right)=\left( \begin{array}{c}  I_n \\ -\nabla_{xx}^2
V(t_0,x_0) \end{array} \right).
\end{equation}
Under the current assumptions, $Y(t)$ is invertible  and
$R(t)=Q(t)Y(t)^{-1}$  for $t\in[t_0,a]$.

Let $\Psi:[t_0,T] \rightarrow M(2n)$ be the solution of the linear system
\begin{equation*}
\dot{\Psi} = \begin{bmatrix}\; H_{xp}[t] & \; H_{pp}[t] \\ - H_{xx}[t] & - H_{px}[t] \end{bmatrix} \Psi,\quad  \Psi(t_0)=I_{2n}.
\end{equation*}
By Lemma \ref{LemmaTacnico},
\begin{equation}\label{EqBase}
(x(t)-\overline{x}(t),p(t)-\overline{p}(t))= \Psi(t)(y_0-x_0,p(t_0)-\overline{p}(t_0))
+ o_t(\mid y_0-x_0 \mid +\mid p(t_0)- \overline{p}(t_0) \mid).
\end{equation}
Moreover, since $V(t_0,\cdot)$ is twice Fr\'echet differentiable at $x_0$,
\begin{equation}\label{Base3}
p(t_0)-\overline{p}(t_0)=-\nabla_x V(t_0,y_0)+\nabla_x
V(t_0,x_0)=-\nabla_{xx}^2 V(t_0,x_0)(y_0-x_0) + o(\mid y_0-x_0
\mid).
\end{equation}
Hence, the term $o_t(\mid y_0-x_0 \mid +\mid p(t_0)- \overline{p}(t_0) \mid)$ in \eqref{EqBase} can be replaced by $o_t(\mid y_0-x_0 \mid)$. Representing $\Psi(t)$ by
\begin{equation*}
\Psi(t)=
\left( \begin{array}{cc} A_1(t) & A_2(t) \\ A_3(t) & A_4(t) \end{array} \right) ,
\end{equation*}
where $A_i(t)\in M(n)$, for $i=1,...,4$, from \eqref{EqBase} we deduce that
\begin{equation}\label{Base1}
x(t)-\overline{x}(t)= A_1(t)(y_0-x_0)+A_2(t)(p(t_0)-\overline{p}(t_0))
+ o_t(\mid y_0-x_0 \mid),
\end{equation}
and
\begin{equation}\label{Base2}
p(t)-\overline{p}(t)= A_3(t)(y_0-x_0)+A_4(t)(p(t_0)-\overline{p}(t_0))
+ o_t(\mid y_0-x_0 \mid).
\end{equation}
Thanks \eqref{Base3}, \eqref{Base1}, we get
\begin{equation}\label{Base6}
x(t)-\overline{x}(t)= A_1(t)(y_0-x_0)-A_2(t)\nabla_{xx}^2 V(t_0,x_0)(y_0-x_0)
+ o_t(\mid y_0-x_0 \mid).
\end{equation}
The key point now is to observe that the map
$$ t\mapsto \Psi(t) \left( \begin{array}{c} I_n \\ -\nabla_{xx}^2 V(t_0,x_0) \end{array} \right) $$
coincides with the solution $(Y,Q)$ of  system
\eqref{LineForw}. Moreover,
\begin{equation}\label{Base7}
A_1 -  A_2 \nabla_{xx}^2 V(t_0,x_0)= Y, \quad
A_3 - A_4 \nabla_{xx}^2 V(t_0,x_0)= Q.
\end{equation}
Thus, from \eqref{Base6}, \eqref{Base7}, we deduce that
\begin{equation}\label{Base9}
x(t)-\overline{x}(t)= Y(t)(y_0-x_0)+ o_t(\mid y_0-x_0 \mid).
\end{equation}
We already know that
\begin{equation}\label{Base4}
\nabla_x V(t,x(t))= - p(t),\quad \nabla_x V(t,\overline{x}(t)) = -
\overline{p}(t).
\end{equation}
By \eqref{Base3}, \eqref{Base2}, \eqref{Base7}, and \eqref{Base4},
\begin{equation}\label{Base8}
\nabla_x V(t,x(t))-\nabla_x V(t,\overline{x}(t))= -Q(t)(y_0-x_0) +
o_t(\mid y_0-x_0 \mid).
\end{equation}
So, from \eqref{Base9}, \eqref{Base8}, and the invertibility of
$Y(\cdot)$ on $[t_0,a]$, we finally deduce that, on $[t_0,a]$,
\begin{equation*}
\nabla_x V(t,x(t))-\nabla_x V(t,\overline{x}(t)) = -
Q(t)Y(t)^{-1}(x(t) -\overline{x}(t)) + o_t( \mid y_0 - x_0 \mid ).
\end{equation*}
Moreover, thanks to \eqref{Base9} and the fact that $Y(t)$ has a
trivial kernel, the term $o_t ( \mid y_0 - x_0 \mid )$ can be
replaced by $o( \mid x(t) - \overline{x}(t) \mid )$, showing that
\begin{equation}\label{QuasiConl}
\nabla_x V(t,x(t))-\nabla_x V(t,\overline{x}(t)) = -
Q(t)Y(t)^{-1}(x(t) -\overline{x}(t)) + o( \mid x(t) -
\overline{x}(t) \mid ).
\end{equation}
The above equality, together with the fact that $\Gamma_t
(\mathring{B}(x_0,\rho))$ is a neighborhood of $\overline{x}(t)$,
allows to deduce that $V(t,\cdot)$ is twice Fr\'echet
differentiable at $\overline{x}(t)$ for any $t\in [t_0,a]$, and
$-\nabla^{2}_{xx}V(t,\overline{x}(t))= Q(t)Y(t)^{-1}=R(t)$.
\end{proof}
In the corollary below, we point out that the twice Fr\'echet
differentiability of $V(t,\cdot)$ propagates up to the final time
$T$ along the optimal trajectory starting from a point
$(t_0,x_0)$, when the cost $\phi$ is locally semiconcave. In this
case, $V$ is also locally semiconcave and the Hessian
$\nabla_{xx}^2 V(t_0,x_0)$ exists  for a.e.  $x_0 \in
\mathbb{R}^n$. Therefore, such a global propagation property is,
in a sense, generic.
\begin{corollary}
Assume, in addition to the hypotheses of Theorem
\ref{PropForward}, that the cost $\phi$ is a locally semiconcave
function. Then, the interval $[t_0,a]$ in Theorem
\ref{PropForward} can be taken equal to $[t_0,T]$.
\end{corollary}
\begin{proof}
Set $a=\sup\lbrace t\in [t_0,T]:\, R(\cdot) \mbox{ is defined on }
[t_0,t]\rbrace. $ Thus, either $R$ is well defined  on $[t_0,T]$
or $\parallel R(t) \parallel \rightarrow \infty$ as $t \rightarrow
a^-$. By the semiconcavity of $V(t,\cdot)$, a constant $c>0$
exists such that $R(t)= - \nabla^2_{xx} V(t,\overline{x}(t))\geq
-c I_n$ (in the sense of quadratic forms) for all $t\in [t_0,a)$.
On the other hand, the solution $\tilde{R}$ of the linear equation
\begin{equation}\label{LinearPDE}
\left\lbrace
\begin{array}{l}
\dot{\tilde{R}}(t)+ H_{px}[t]\tilde{R}(t)+\tilde{R}(t)H_{xp}[t] + H_{xx}[t]=0\quad \textit{ on }[t_0,T], \\
\tilde{R}(t_0)=-\nabla_{xx}V(t_0,x_0).
\end{array}
\right.
\end{equation}
verifies, as a quadratic form, the inequality $\tilde{R}\geq R$ on
$[t_0,a)$ because $H_{pp}[\cdot]\geq 0$. The fact that the
solution of \eqref{LinearPDE} is well defined on $[t_0,T]$
provides a bound from  above for $R(\cdot)$ on $[t_0,a)$. These
observations and the symmetry of $R(t)$ imply that $R$ is bounded
on $[t_0,a)$, excluding the possibility of a blow-up at time $a$.
So, the solution to \eqref{DeivSecBase} is defined on $[t_0,T]$.
\end{proof}
We say that $Q\in M(n)$ is a \emph{reachable Hessian} of $V(t,\cdot)$ at $x$ if a sequence $\lbrace x_j\rbrace \subset \mathbb{R}^n$ converging to $x$ exists such that $V(t,\cdot)$ is twice Fr\'echet differentiable at $x_j$, and $Q= \lim_{j\rightarrow \infty} \nabla^2_{xx} V(t,x_j).$ The set of all reachable Hessians of $V(t,\cdot)$ at $x$ is denoted by $\partial^{*}_{xx} V(t,x)$. By construction, $\partial^{*}_{xx} V(t,x)$ is a subset of $S(n)$. 
The question naturally arises whether we can guarantee the forward propagation of the solution of the Riccati equation in the set of all reachable Hessians of $V(t,\cdot)$ along an optimal trajectory.
\begin{proposition}
Assume $(SH)$,  and suppose  $H \in C^{2} (\mathbb{R}^n \times ( \mathbb{R}^n
\setminus \lbrace 0\rbrace))$ and $\phi$ is locally
Lipschitz with $\partial^+ \phi(z)\neq \emptyset$ for all $z\in
\mathbb{R}^n$. Let $\overline{x}$ be an optimal solution for
$\mathcal{P}(t_0,x_0)$ such that
 $\partial^+ \phi(\overline{x}(T))\neq \lbrace 0 \rbrace$ and consider a dual arc $\overline{p}$ satisfying $ -\overline{p}(T)\in \partial^+ \phi(\overline{x}(T))\setminus \lbrace 0 \rbrace $.
If $V(t_0,\cdot)$ is of class $C^{1}$ in a neighborhood of $x_0$
and  for some $a\in (t_0, T]$ and $Q\in \partial^{*}_{xx}
V(t_0,x_0)$, the Riccati equation
\begin{equation}\label{DeivSec}
\left\lbrace
\begin{array}{l}
\dot{R}(t)+ H_{px}[t]R(t)+R(t)H_{xp}[t] + R(t)H_{pp}[t]R(t) + H_{xx}[t]=0, \\
R(t_0)=- Q,
\end{array}
\right.
\end{equation}
has a solution $R(\cdot)$ defined on $[t_0,a]$, then, for all $ t\in[t_0,a]$,
\begin{equation}\label{IncluSecondReach}
R(t)\in - \partial^{*}_{xx} V(t,\overline{x}(t)).
\end{equation}
Moreover, if $\phi$ is a locally semiconcave function, then
$R(\cdot)$ is well defined on $[t_0,T]$, and so \eqref{IncluSecondReach} is true
on the whole interval $[t_0,T]$.
\end{proposition}
\begin{proof}
 Consider a sequence $x_j^0 \rightarrow x_0$ such that $\lim_{j\rightarrow \infty} \nabla_{xx}^2 V(t_0,x_j^0)=Q$.
  Let $x_j(\cdot)$ be optimal for $\mathcal{P}(t_0,x_j^0)$ and  $p_j(\cdot)$ be a dual arc corresponding to
   $x_j(\cdot)$ with $-p_j (T)\in \partial^+ \phi (x_j(T))$. It is not difficult to verify that for all large $j$, $(x_j,p_j)$ is uniquely defined.
    Setting $H_{px}^j[t]:=\nabla^2_{px} H (x^j(t),p^j(t))$, we have that $H_{px}^j[\cdot]$ converges to $H_{px}[\cdot]$
    uniformly on $[t_0,T]$, and similarly for the other partial Hessians of $H$. Let $R_j$
denote the solution of  \eqref{DeivSec} with $Q$ replaced by
$\nabla_{xx}^2 V(t_0,x_j^0)$, $H_{px}[t]$ by $H_{px}^j[t]$, and
similarly for the other partial Hessian of $H$. Then, from a
standard argument based on Filippov's Theorem (see, e.g., Theorem
10.4.1 in \cite{MR1048347}), we deduce that $R_j$ are well-defined
on $[t_0,a]$ for $j$ large enough, and they converge uniformly to
$R$ on $[t_0,a]$. Applying Theorem \ref{PropForward} and passing
to limit, we conclude that \eqref{IncluSecondReach} holds true.
\end{proof}
\subsection{Backward propagation of twice Fr\'echet differentiability}
 In general, the differentiability of $V(t,\cdot)$ does not propagate along
optimal trajectories
 backward in time. However,  this happens when  optimal solutions are unique. Then the twice Fr\'chet differentiability
 propagates
 backward in time,
  as long as there is no conjugate time.
\begin{theorem}\label{Back_prop}
Assume $(SH)$,  and suppose $H \in C^{2} (\mathbb{R}^n \times ( \mathbb{R}^n
\setminus \lbrace 0\rbrace))$ and $\phi \in C^{1,1}_{loc}$.
Suppose that for some $(t_0,x_0) \in (-\infty,T)\times
\mathbb{R}^n$ and $r>0$, $\mathcal{P}(t_0,y_0)$ has a unique
optimal trajectory for every $y_0 \in B(x_0,r)$. Let
$\overline{x}$ be the optimal solution of $\mathcal{P}(t_0,x_0)$
such that $\nabla \phi (\overline{x}(T))\neq 0$.  If $\phi$ is
twice Fr\'echet differentiable at $\overline{x}(T)$ and  the
solution $R$ of the Riccati equation
\begin{equation}\label{Back_Riccati}
\left\lbrace
\begin{array}{l}
\dot{R}(t)+ H_{px}[t]R(t)+R(t)H_{xp}[t] + R(t)H_{pp}[t]R(t) + H_{xx}[t]=0, \\
R(T)=-\nabla^2 \phi (\overline{x}(T)),
\end{array}
\right.
\end{equation}
is defined on $[a,T]$, for some $a\in [ t_0, T)$, then $V(t,\cdot)$ is 
twice Fr\'echet differentiable at $\overline{x}(t)$ for any $t\in
[a,T]$. Moreover, $R(t)=- \nabla^2_{xx} V(t,\overline{x}(t))$.
\end{theorem}
\begin{proof}
Let $\epsilon >0$ be such that $\nabla \phi (x)\neq 0$ for all $x
\in B(\bar x(T),\epsilon)$.  By \cite[Theorem~4.1]{MR2728465}),
$V$ is locally semiconcave.
 We claim that
$V(t_0,\cdot)$ is continuously differentiable on a neighborhood of
$x_0$. Indeed, from the uniqueness assumption it is not difficult
to deduce that for some $r_1 \in (0,r)$,  the optimal trajectory
$y:[t_0,T]\rightarrow \mathbb{R}^n$ for  problem
$\mathcal{P}(t_0,y_0)$ with $y_0 \in B(x_0,r_1)$ satisfies $y(T)
\in B(\bar x(T),\epsilon)$ and  admits the unique dual arc
$p:[t_0,T]\rightarrow \mathbb{R}^n$.  Then  the pair $(y, p)$
solves the system
\begin{equation}\label{hamy}
\left\{\begin{array}{rlrrrl}
\dot{x}(t) &=& \nabla_p H(x(t),p(t)), &  x(T)&=&x_T, \\
-\dot{p}(t) &=& \nabla_x H(x(t),p(t)),& -p(T)&=&\nabla \phi (x_T),
\end{array}\right.
t\in \left[ t_0, T \right],
\end{equation}
for $x_T=y(T)$.
   If $V(t_0,\cdot)$ is differentiable at $y_0 \in B(x_0,r_1)$, then $p(t_0)=-\nabla_xV(t_0,y_0)$.  Using  arguments
similar to those of the proof of Proposition \ref{TheoReach} we
check that $0 \notin \partial^{*}_x V(t_0,y_0)$  for every $y_0
\in \mathring{B}(x_0,r_1)$. This and Proposition \ref{TheoReach}
imply that $\partial^{*}_x V(t_0,y_0)$ is a singleton for every
such $y_0$ and our claim is proved.

By Theorem \ref{PropForward}, $V(t,\cdot)$ is continuously
differentiable on a neighborhood of $\bar x(t)$  for all $t \in
[t_0,T]$.

Fix $t\in[a,T]$ and let $r_2 > 0$ be such that $V(t,\cdot)$ is
continuously differentiable on $B(\bar x(t),r_2)$ and
$\nabla_xV(t,y_0) \neq 0$ for all $y_0 \in B(\bar x(t),r_2)$.
Consider the Hamiltonian system
\begin{equation}\label{Variation2}
\left\{
\begin{array}{rlrl}
\dot{x}(s)&= \nabla_p H(x(s),p(s)),& x(t)&=y_0, \\
-\dot{p}(s)&= \nabla_x H(x(s),p(s)),& p(t)&= -\nabla_xV(t,y_0).
\end{array}\right.
\end{equation}
and denote its solution  by $(x(\cdot;y_0),p(\cdot;y_0))$. As in
the proof of Theorem \ref{PropForward} we check that
 the map $\Gamma_t (y_0):=x(T,y_0)$
 is a homeomorphism from $\mathring{B}(\overline{x}(t),r_2)$
onto the open set $\Gamma_t (\mathring{B}(\overline{x}(t),r_2) )$.
Let $\delta >0$ be such that $B(\bar x(T),\delta) \subset \Gamma_t
(\mathring{B}(\overline{x}(t),r_2) )$. To prove that $V(t,\cdot)$
is twice Fr\'echet differentiable at $\overline{x}(t)$ consider
the solution $(X,P)$ to the variational system:
\begin{equation}\label{CP*}
\left\{\begin{array}{rllrrl}
\dot{X}(s)&=&H _{xp} [s]X(s) + H_{pp}[s] P(s) ,& X(T)&=& I, \\
-\dot{P}(s)&=& H_{xx}[s] X(s)+ H_{px} [s] P(s),&
-P(T)&=&\nabla^2\phi(\overline{x}(T)),
\end{array}\right.
\end{equation}
on $ [t_0 , T ]$, and the solution $(x,p)$ to \eqref{hamy} with
$x_T \in B(\overline{x}(T),\delta)$.  Since $\nabla \phi$ is
locally Lipschitz and Fr\'echet differentiable at $\bar x(T)$,
\begin{equation}\label{impo}
x(t)-\overline{x}(t)= X(t)(x(T)- \overline{x}(T))+ o_t(\mid x(T)- \overline{x}(T) \mid),\quad \forall t\in [t_0,T],
\end{equation}
and
\begin{equation}\label{A1}
p(t)-\overline{p}(t)= P(t)(x(T)- \overline{x}(T))+ o_t(\mid x(T)- \overline{x}(T) \mid),\quad \forall t\in [t_0,T].
\end{equation}
Moreover, $X$ is invertible on $[a,T]$, and $R= P X^{-1}$ on the same interval. Note that
\begin{equation}\label{A2}
p(t)-\overline{p}(t)=- (\nabla_x V(t,x(t))-  \nabla_x
V(t,\overline{x}(t))).
\end{equation}
From \eqref{impo}-\eqref{A2} and the invertibility of $X$ on
$[a,T]$, we get
\begin{equation*}
\nabla_x V(t,x(t))- \nabla_x V(t,\overline{x}(t))= - P(t)X(t)^{-1}(x(t)- \overline{x}(t))+ o_t(\mid x(t)- \overline{x}(t) \mid),\quad \forall t\in [a,T].
\end{equation*}
Since $\Gamma_t^{-1}(\mathring{B}(\bar x(T),\delta))$ is an open
neighborhood of  $\bar x(t) $,   the proof is complete.
\end{proof}

\subsection{Local $C^2$ regularity of the value function}
In the theorem below we shall prove that the existence of a
nonzero proximal subgradient of $V(t_0,\cdot)$ at $x_0$ is
sufficient for the $C^{2}$ regularity of $V$ in a neighborhood
$(t_0, x_0)$, as well as for its $C^{2}$ regularity in a
neighborhood of an optimal trajectory starting from such point.
The proof is based upon ideas from \cite{frankowska:hal-00851752},
where the first two authors of this paper have investigated the
Bolza problem in the calculus of variations assuming the
Hamiltonian to be strictly convex in the $p$ variable. However, in
our context this assumption cannot be satisfied; indeed, $H$ is
positively homogeneous in $p$ and $\nabla^2_{pp} H(x,p)p=0$ for
all $p\neq 0$, whenever $\nabla^2_{pp} H(x,p)$ exists. The proof
proposed below simplifies the argument in
\cite{frankowska:hal-00851752}. Note that here we cannot expect
the same result to be valid for null proximal subgradients, since
the Hamiltonian is singular at $p=0$.
\begin{theorem}\label{RegularityValueFunction}
Assume $(SH)$,   and suppose $H \in C^{2} (\mathbb{R}^n \times ( \mathbb{R}^n
\setminus \lbrace 0\rbrace))$ and $\phi \in C^{2}(\mathbb{R}^n)$.
Let $(t_0,x_0)\in (-\infty,T]\times\mathbb{R}^n$ and let
$\overline{x}$ be an optimal solution for $\mathcal{P}(t_0,x_0)$
such that $\nabla\phi(\overline{x}(T))\neq 0$. If $\partial_x
^{-,pr} V(t_0,x_0)\neq \emptyset$, then $V(t,\cdot)$ is of class
$C^{2}$ in a neighborhood of $\overline{x}(t)$ for all $t \in
[t_0,T]$. Moreover, if $H$ is of class
$C^{2,1}_{loc}(\mathbb{R}^n\times (\mathbb{R}^n \smallsetminus
\lbrace 0 \rbrace))$ and, for some $0 < m \leq 1$, $\phi$ is of
class $C^{2,m}_{loc}(\mathbb{R}^n)$, then $V(\cdot,\cdot)$ is of
class $C^{2,m}$ in a neighborhood of $(t,\overline{x}(t))$ for all
$t\in [t_0,T)$.
\end{theorem}
\begin{proof} Set $z_0=\bar x(T).$
From Theorem \ref{TheoDualArc2} and Remark \ref{RemarkDualArc} we deduce that
there exists a nonvanishing arc $\overline{p}: [t_0,T]\rightarrow
\mathbb{R}^n$ such that
$(\overline{x}(\cdot),\overline{p}(\cdot))$ solves the system:
\begin{equation}\label{CJ}
\left\{\begin{array}{rllrrl}
\dot{x}(t)&=& \nabla_p H (x(t),p(t)), & x(T)&=&z_0,\\
-\dot{p}(t)&=& \nabla_x H (x(t),p(t)), & p(T)&=&\nabla \phi(z_0).
\end{array}\right. t \in
[t_0,T]
\end{equation}
Since $V$ is locally semiconcave (see
\cite[Theorem~4.1]{MR2728465}) and $\partial^{-,pr}_x
V(t_0,x_0)\neq \emptyset$,  $V(t_0,\cdot)$ is Fr\'echet
differentiable at
 $x_0$, and $\partial^{-,pr}_x V(t_0,x_0)=\lbrace \nabla_x V(t_0,x_0)\rbrace $. Moreover, from Corollary \ref{Differentiability} we deduce
 that $V(t,\cdot)$ is differentiable at $\overline{x}(t)$ with $\nabla_x V (t,\overline{x}(t))= - \overline{p}(t)$ for all $t \in [t_0,T]$.
 Note also that $\lbrace -\overline{p}(t_0)\rbrace = \partial^{-,pr}V(t_0,x_0)$. Let $(X(\cdot),P(\cdot)):[t_0,T]\rightarrow M (n) \times M (n)$ be
 the solution of the following system: for $t\in [t_0,T]$,
\begin{equation}\label{CP}
\left\{\begin{array}{rllrrl}
\dot{X}(t) &=&H_{xp}[t] X(t) + H_{pp}[t] P(t) ,& X(T)&=& I, \\
-\dot{P}(t)&=& H_{xx}[t] X(t) + H_{px}[t] P(t) ,& -P(T)&=&\nabla^2
\phi(z_0).
\end{array}\right.
\end{equation}
 We prove that there is no conjugate time in $[t_0,T]$ for $z_0$. Let us argue by contradiction and suppose that a conjugate time  $t_c \in [t_0,T]$ does exist.
 Thus, for some $\theta\in S^{n-1}$,  $X(t_c)\theta=0$. Then $P(t_c)\theta\neq 0$, by the uniqueness of the solution of \eqref{CP}.
 Furthermore, for all $t>t_c$, $V(t,\cdot)$ is of class $C^{2}$ in a neighborhood of $\overline{x}(t)$ and $\nabla^2_{xx}V(t,\overline{x}(t))= - P(t)X(t)^{-1}$ by Theorem \ref{TheoPntConju}. This yields
\begin{equation}\label{limit}
\parallel \nabla^2_{xx}V(t,\overline{x}(t))\parallel\geq \frac{|P(t)\theta|}{|X(t)\theta|}\rightarrow +\infty \mbox{ as } t\rightarrow t_c^+.
\end{equation}
Our goal is to show that the above limit cannot be infinite. First
note that, since $V(t,\cdot)$ is locally semiconcave, its second
derivatives at $\overline{x}(t)$ are bounded from above uniformly
for $t>t_c$, in the sense that for some constant $c>0$ and all $T
\geq t > t_c$,
\begin{equation} \label{alto}
 \langle \nabla^2_{xx}V (t,\overline{x}(t)) \zeta , \zeta \rangle\leq c \ \mbox{ for all } \zeta \in S^{n-1}.
\end{equation}
On the other hand, since $- \overline{p}(t_0)\in \partial^{-,pr}_x V(t_0,x_0)$, from Theorem \ref{Lemma_sub_prox} we know that, for some $c_0\geq 0$ and $r_0>0$, and for all $t\in [t_0,T]$ and $h\in B(0,r_0)$,
\begin{equation}\label{1}
-c_0 |h|^2 \leq V(t,\overline{x}(t)+h)-V(t,\overline{x}(t)) -\langle \nabla_x V(t,\overline{x}(t)), h \rangle .
\end{equation}
For $t>t_c$ we have
\begin{equation}\label{2}
V(t,\overline{x}(t)+h)-V(t,\overline{x}(t)) -\langle \nabla_x V(t,\overline{x}(t)), h \rangle = \frac{1}{2}\langle \nabla^2_{xx}V(t,\overline{x}(t))h,h)+o(\mid h \mid^2).
\end{equation}
Now, combining \eqref{1} and \eqref{2} with the choice of $h=\tau
\zeta,~\zeta\in S^{n-1}, $  taking $\tau>0$  sufficiently small,
dividing by $\tau^2$, and passing to the limit as $\tau \rightarrow
0$ we get that, for all $T \geq t > t_c$,
\begin{equation}\label{basso}
-c_0 \leq \frac{1}{2} \langle \nabla^2_{xx}V (t,\overline{x}(t)) \zeta , \zeta \rangle~\ \mbox{ for all } \zeta \in S^{n-1}.
\end{equation}
Since $\nabla^2_{xx}V (t,\overline{x}(t))$ is symmetric, from \eqref{alto} and \eqref{basso} we deduce that for all $t\in (t_c,T]$
\begin{equation*}
\parallel \nabla^2_{xx}V (t,\overline{x}(t)) \parallel = \max_{\beta \in S^{n-1}} |\langle \beta, \nabla^2_{xx}V(t,\overline{x}(t))\beta \rangle |\leq \max(c_0,c),
\end{equation*}
where the constants $c,c_0$ are independent of $t\in(t_c,T]$. This
means that $\nabla^2_{xx}V(t,\overline{x}(t))$ is bounded on $
(t_c,T]$, in
 contradiction with \eqref{limit}. Thus, the interval $[t_0,T]$ does not contain any conjugate time for $z_0$. Since $\phi$ is of class $C^{2}$, the absence of
  conjugate times is equivalent to the fact that $V(t,\cdot)$ is $C^{2}$ in some neighborhood of $\overline{x}(t)$ for every $t \in [t_0,T]$, thanks to
  Theorem \ref{TheoPntConju}. Moreover, if the cost and the Hamiltonian are of class $C_{loc}^{2,m} (\mathbb{R}^n)$ and $C_{loc}^{2,1}(\mathbb{R}^n \times
  (\mathbb{R}^n\smallsetminus \lbrace 0\rbrace) )$, respectively, then  Theorem \ref{TheoPntConju} guarantees the $C^{2,m}$ regularity of $V(t,\cdot)$ in a
  neighborhood of $\overline{x}(t)$ for any $t \in [t_0,T]$. At this stage, one can prove the $C^{2,m}$ regularity of $V(\cdot,\cdot)$ in a neighborhood of
   $(t,\overline{x}(t))$ taking advantage of the fact that $V$ satisfies the Hamilton-Jacobi equation \eqref{HJB}. The proof of this step follows very closely the last part
    of the proof of Theorem 4.2 in \cite{frankowska:hal-00851752} and  is omitted.
\end{proof}
\begin{remark}
The hypothesis that $\nabla\phi(\overline{x}(T))\neq 0$ plays an
important role in the proof of the above theorem, its conclusion
would be false
 otherwise. We illustrate this fact with the following example. Let us consider  the final cost
 $\phi:\mathbb{R}\rightarrow\mathbb{R},\; \phi(z) = z^2$ and the constant set-valued map $F(x)= [-1,1]$ for $x \in \mathbb{R}$.
The associated  Hamiltonian is $H(p)=\mid p \mid$, which
satisfies the assumptions of the previous theorem.
 When $\mid x \mid \leq T-t$, any of the admissible trajectories ending at $0$ is optimal. If $x\geq T-t$, then the unique optimal trajectory from
$(t,x)$ is given by
 $y(s;t,x)= x + t - s$, for $s\in [t,T]$. Analogously, if $x\leq t-T$, the optimal trajectory is $y(s;t,x)= x- t + s$.
 We deduce that the value function is as follows: for $(t,x)\in [0,T]\times \mathbb{R}$,
\[
V(t,x)=
\left\lbrace
\begin{array}{ll}
(x + t - T)^2, & \textit{if } x \geq T - t ,\\
(x - t + T)^2, & \textit{if } x \leq t - T ,\\
0,  & \textit{otherwhise.}
\end{array}
\right.
\]
$V$ is not twice differentiable at $(t,x)$ with $x=\pm( T-t)$,
which are also the points with $ \phi'(y(T;t,x))=0$. Moreover, for such points $(t,x)$ the set $\partial_x
^{-,pr} V(t,x)$ is nonempty, since it is the singleton $\lbrace 0
\rbrace$. Hence, the fact that the proximal subgradient of
$V(t,\cdot)$ at $x$ contains zero (or, equivalently in our
context, $\nabla \phi (\overline{x}(T))=0$) does not guarantee
that $V(t,\cdot)$ is twice differentiable at $x$.
\end{remark}
\begin{example}
Here is a simple example of a system with an Hamiltonian of class $C^2 (\mathbb{R}^n \times ( \mathbb{R}^n \setminus \lbrace 0 \rbrace ))$. Consider a dynamic ${\dot x}= f(x,u)$ given by the \emph{affine} in control system:
\[ f(x)= h(x)+ g(x)u, \; u\in U \]
where $U\subset \mathbb{R}^m$ is the closed unit ball, $m \geq n$.
Suppose that $h:\mathbb{R}^n\rightarrow \mathbb{R}^n$ and
$g:\mathbb{R}^n \rightarrow \mathbb{R}^{n \times m}$ are of class
$C^2$, and that the matrix $g(x)$ has full rank for all
$x\in\mathbb{R}^n$. Then $ H(x,p)= \langle p,h(x)\rangle +  \mid
g(x)^* p \mid , $ which is clearly of class $C^2 (\mathbb{R}^n
\times ( \mathbb{R}^n \setminus \lbrace 0 \rbrace)).$  Similarly
one can consider strictly convex sets $U$ with sufficiently smooth
boundary.
\end{example}
\section*{Acknowledgements}
This work was partially supported by the European Commission
(FP7-PEOPLE-2010-ITN, Grant Agreement no.\ 264735-SADCO), and by
 the ``Instituto Nazionale di Alta Matematica'' (INdAM), through the GNAMPA Research Project 2014: ``Controllo moltiplicativo per modelli diffusivi nonlineari''.

\printbibliography

\end{document}